\documentclass[11pt]{amsart}
\usepackage[colorlinks=true,citecolor=blue,linkcolor=blue]{hyperref}
\usepackage{amssymb,amscd,amsmath,graphicx,mathtools}
\usepackage{enumerate}
\usepackage{fullpage}

\usepackage[all]{xy}
\usepackage{cleveref}

\numberwithin{equation}{section}
\newtheorem{theorem}{Theorem}[section]
\newtheorem{lemma}[theorem]{Lemma}
\newtheorem{proposition}[theorem]{Proposition}
\newtheorem{corollary}[theorem]{Corollary}

\theoremstyle{definition}
\newtheorem{definition}[theorem]{Definition}

\newtheorem*{acknowledgment}{Acknowledgments}

\newtheorem{question}[theorem]{Question}

\theoremstyle{definition}
\newtheorem{remark}[theorem]{Remark}
\newtheorem{example}[theorem]{Example}

\DeclareMathOperator{\chara}{char}

\DeclareMathOperator{\NC}{NC}

\DeclareMathOperator{\RV}{RV}

\DeclarePairedDelimiter\ceil{\lceil}{\rceil}

\newcommand{\PP}{{\mathbb P}}
\newcommand{\ZZ}{{\mathbb Z}}
\newcommand{\NN}{{\mathbb N}}

\newcommand{\RR}{{\mathbb R}}

\def\R{{\mathcal R}}

\def\R{{\mathcal R}}

\def\aa{{\mathfrak a}}
\def\mm{{\mathfrak m}}
\def\nn{{\mathfrak n}}

\def\qq{{\mathfrak q}}

\def\bb{{\mathfrak b}}
\def\cc{{\mathfrak c}}

\newcommand{\kk}{\Bbbk}

\def\ahat{\widehat{\alpha}}
\def\rhat{\widehat{\rho}}
\def\rlim{\rho^{\lim}}

\def\vhat{\widehat{v}}

\def\1{{\bf 1}}
\def\0{{\bf 0}}

\begin{document}
	
	\title{Resurgence number of graded families of ideals}
	
	\author{T\`ai Huy H\`a}
	\address{Tulane University, Department of Mathematics, 6823 St. Charles Ave., New Orleans, LA 70118, USA}
	\email{tha@tulane.edu}
	%\urladdr{???}
	
%	\author{A.V. Jayanthan}
%	\address{Department of Mathematics, I.I.T. Madras, Chennai - 600036, INDIA}
%	\email{jayanav@iitm.ac.in}
	
	\author{Arvind Kumar}
	\address{Department of Mathematics, Chennai Mathematical Institute, Siruseri
		Kelambakkam, Chennai, India - 603103}
	\email{arvindkumar@cmi.ac.in}
	
	\author{Hop D. Nguyen}
	\address{Institute of Mathematics, VAST, 18 Hoang Quoc Viet, Cau Giay, 10307 Hanoi, Vietnam}
	\email{ngdhop@gmail.com}

        \author{Th\'ai Th\`anh Nguy$\tilde{\text{\^E}}$n}
        \address{Department of Mathematics and Statistics, McMaster University, Canada and University of Education, Hue University, 34 Le Loi St., Hue, Viet Nam}
        \email{nguyt161@mcmaster.ca  tnguyen11@tulane.edu}
	
	
	\begin{abstract}
		We define the resurgence and asymptotic resurgence numbers associated to a pair of graded families of ideals in a Noetherian ring. These notions generalize the well-studied resurgence and asymptotic resurgence of an ideal in a polynomial ring. We examine when these invariant are finite and rational. We investigate situations where these invariant can be computed via Rees valuations or realized as actual limits of well-defined sequences. We study how the asymptotic resurgence changes when a family is replaced by its integral closure. Many examples are given to illustrate that whether or not known properties of resurgence and asymptotic resurgence of an ideal would extend to that of a pair of graded families of ideals generally depends on the Noetherian property and finite generation of the Rees algebras of these families.
	\end{abstract}
	
	\maketitle

%%%%%%%%%%%%%%%%%%%%%%%%%%%%%%%%%%
%%%%%%%%%%%%%%%%%%%%%%%%%%%%%%%%%%

\tableofcontents

\section{Introduction} \label{sec.intro}

The \emph{Ideal Containment Problem}, which investigates when symbolic powers of an ideal are contained in its ordinary powers, originated from pioneer work of Ein, Lazarsfeld and Smith \cite{ELS} and of Hochster and Huneke \cite{HH}, and has evolved to be an active research area during the last two decades (cf. \cite{BGHN22a, BGHN22b, BN21, BH, CEHH, DS2020, DT2017, G2020, HH2013, HS2015, TX2020} and references therein and thereafter). In this research program, resurgence and asymptotic resurgence numbers were introduced as measures for the non-containment between symbolic powers and ordinary powers of an ideal. These invariant have attracted much attention and been studied by many authors (cf. \cite{BGHN22a, BH, CHHVT, DFMS, DHNSTT, GHM2020, GHV, JKM2021, Ng23a, Ng23b}).

In this paper, we generalize these notions to measure the non-containment between members of arbitrary graded families of ideals. Let $S$ be a Noetherian commutative ring, and let $\aa_\bullet = \{\aa_i\}_{i \ge 1}$ and $\bb_\bullet = \{\bb_i\}_{i \ge 1}$ be graded families of ideals in $S$. We define the \emph{resurgence} and \emph{asymptotic resurgence} numbers of the ordered pair $(\aa_\bullet, \bb_\bullet)$ to be:
\begin{align*}
	\rho(\aa_\bullet, \bb_\bullet) & = \sup\left\{ \dfrac{s}{r} ~\Big|~ s, r \in \NN, \aa_s \not\subseteq \bb_r \right\}, \text{ and } \\
	\rhat(\aa_\bullet,\bb_\bullet) & = \sup\left\{ \dfrac{s}{r} ~\Big|~ s, r \in \NN, \aa_{st} \not\subseteq \bb_{rt} \text{ for } t \gg 1 \right\}.
\end{align*}
Obviously, if $\aa_\bullet = \{I^{(i)}\}_{i \ge 1}$ and $\bb_\bullet = \{I^i\}_{i \ge 1}$ are families of symbolic and ordinary powers of an ideal $I \subseteq S$, then we recover the usual resurgence and asymptotic resurgence numbers of $I$, which were defined by Bocci and Harbourne \cite{BH} and by Guardo, Harbourne and Van Tuyl \cite{GHV}, respectively. Furthermore, investigating the resurgence and asymptotic resurgence numbers associated to a pair of graded families of ideals allows us to look at the containment problem from a much more general context. For instance, this would include the notion of \emph{integral closure resurgence}, which was introduced by Harbourne, Kettinger and Zimmitti \cite{HKZ}, and the resurgence and asymptotic resurgence of a filtration associated to a covering polyhedron \emph{relative} to another filtration of ideals, which was recently studied by Grisalde, Seceleanu and Villarreal \cite{GSV2021}. This approach would also allow us to consider containment between a graded family of ideals and a filtration of tight closures of powers of another ideal in positive characteristics.

A priori, from their definitions, it is difficult to compute the resurgence and asymptotic resurgence numbers. In fact, to the best of our knowledge, algorithms to compute the resurgence number of an ideal in general are not available. On the other hand, DiPasquale, Francisco, Mermin and Schweig \cite{DFMS} showed that, when $S$ is a polynomial ring over a field, the asymptotic resurgence number $\rhat(I)$ of an ideal $I \subseteq S$ can be computed by considering Rees valuations of $I$, and that $\rhat(I) = \rhat\left(I^{(\bullet)}, \overline{I^\bullet}\right)=\rho\left(I^{(\bullet)}, \overline{I^\bullet}\right)$, where $ \overline{I^\bullet}$ denotes the filtration $\{\overline{I^n}\}_{n \ge 1}$ of integral closures of powers of $I$. We shall illustrate by many examples that this is not the case for arbitrary pairs $(\aa_\bullet, \bb_\bullet)$ of graded families of ideals. The failure appears to lie at the non-Noetherian property of the Rees algebra $\R(\bb_\bullet)$ of $\bb_\bullet$ and the non-finitely generation of $\R(\overline{\bb_\bullet})$ over $\R(\bb_\bullet)$. Here, the \emph{Rees algebra} $\R(\bb_\bullet)$, which plays a prominent role in this paper,
is defined by
\[
\R(\bb_\bullet)=S\oplus \bb_1t\oplus \bb_2t^2\oplus \cdots \subseteq S[t],
\]
and $\overline{\bb_\bullet}=\{\overline{\bb_i}\}_{i\ge 1}$ is the graded family of integral closures of $\bb_\bullet$.

Our goal is to see for which graded families $\aa_\bullet$ and $\bb_\bullet$ of ideals in a Noetherian commutative ring $S$, the asymptotic resurgence $\rhat(\aa_\bullet, \bb_\bullet)$ can still be computed by Rees valuations, and the equality $\rhat(\aa_\bullet, \bb_\bullet) = \rhat(\aa_\bullet, \overline{\bb_\bullet}) = \rho(\aa_\bullet, \overline{\bb_\bullet})$ remains to hold. In addition, we will define new sequences whose limits realize the asymptotic resurgence numbers $\rhat(\aa_\bullet, \bb_\bullet)$ and $\rhat(\aa_\bullet, \overline{\bb_\bullet})$. We shall further discuss when resurgence and asymptotic resurgence numbers are finite and are rational numbers.

We will now describe main results of the paper. We start by assuming that $S$ is a domain and letting $K$ denote its quotient field. For a valuation $v$ of the $K$, set
$$v(I) = \min \{v(x) ~\big|~ x \in I \setminus \{0\}\}.$$
It can be seen that for any graded family $\aa_\bullet$ of ideals in $S$, $\{v(\aa_i)\}_{i \ge 1}$ is a sub-additive sequence. Thus, by Fekete's Lemma, the limit $\lim\limits_{n \rightarrow \infty} \dfrac{v(\aa_n)}{n}$ exists and is equal to $\inf\limits_{n \in \NN} \dfrac{v(\aa_n)}{n}$. Following \cite{DFMS}, we define the \emph{skew Waldschmidt constant} of $\aa_\bullet$ with respect to $v$ to be
$$\vhat(\aa_\bullet) = \lim\limits_{n \rightarrow \infty} \dfrac{v(\aa_n)}{n} = \inf\limits_{n \in \NN} \dfrac{v(\aa_n)}{n}.$$
Our first result generalizes \cite[Theorem 1.2.1]{BH} from the Waldschmidt constant of an ideal to the skew Waldschmidt constant of a graded family of ideals in $S$ with respect to any valuation of $K$.

\medskip

\noindent\textbf{Theorem \ref{low_bou}.} Let $\aa_\bullet$ and $\bb_\bullet$ be graded families of nonzero ideals in $S$. Let $v$ be a valuation of $K$, that is supported on $S$, such that $\vhat(\bb_\bullet) > 0$. Then,
$$\dfrac{\vhat(\bb_\bullet)}{\vhat(\aa_\bullet)} \le \rhat(\aa_\bullet, \overline{\bb_\bullet}).$$
Moreover, if $\vhat(\aa_\bullet) = 0$ then $\rhat(\aa_\bullet, \overline{\bb_\bullet}) = \infty.$

\medskip

It is known (cf. \cite{SH2006}) that for any ideal $I \subseteq S$, there are finitely many \emph{Rees valuations} associated to $I$. Let $\RV(I)$ denote this finite set of Rees valuations of $I$. Our next result generalizes \cite[Proposition 4.2 and Theorem 4.10]{DFMS}.

\medskip

\noindent\textbf{\Cref{thm.ReesVal} and \Cref{asym_res_int}.} Let $\aa_\bullet$ and $\bb_\bullet$ be graded families of nonzero ideals in $S$.
\begin{enumerate}
	\item Suppose that the Veronese subring $\R^{[k]}(\bb_\bullet)$ of the Rees algebra $\R(\bb_\bullet)$ is a standard graded $S$-algebra for some $k \in \NN$. Then,
$$\rhat(\aa_\bullet, \overline{\bb_\bullet}) = \max_{v \in \RV(\bb_k)} \left\{ \dfrac{\vhat(\bb_\bullet)}{\vhat(\aa_\bullet)} \right\} = \sup_{v(\bb_k) > 0} \left\{\dfrac{\vhat(\bb_\bullet)}{\vhat(\aa_\bullet)}\right\}.$$
Particularly, if $\R^{[\ell]}(\aa_\bullet)$ is a standard graded $S$-algebra for some $\ell$, then $\rhat(\aa_\bullet, \overline{\bb_\bullet})$ is either a positive rational number or infinite.
\item Suppose that $\aa_\bullet$ and $\bb_\bullet$ are filtration and that $\R(\overline{\bb_\bullet})$ is a finitely generated $\R(\bb_\bullet)$-module. Then,
$$\rhat(\aa_\bullet, \overline{\bb_\bullet}) = \rhat(\aa_\bullet, \bb_\bullet).$$
\end{enumerate}

\medskip

Examples exist to show that the conclusions of \Cref{thm.ReesVal} and \Cref{asym_res_int} are not necessarily true without the stated hypotheses (see  Examples \ref{ex.irrational}, \ref{ex.rhatNOTnoeth} and \ref{ex.rhatNOTbar}). Examples also exist to illustrate that even when $\rhat(\aa_\bullet,\bb_\bullet) = \rhat(\aa_\bullet,\overline{\bb_\bullet})$ and the Rees algebra $\R(\bb_\bullet)$ is Noetherian, this value is not necessarily equal to $\rho(\aa_\bullet,\overline{\bb_\bullet})$ (see Example \ref{ex.rhoNOTequalBAR}).

Our next result addresses the question of when the value $\rhat(\aa_\bullet, \bb_\bullet) = \rhat(\aa_\bullet, \overline{\bb_\bullet})$ is equal to $\rho(\aa_\bullet, \overline{\bb_\bullet})$; that is, when replacing the graded family $\bb_\bullet$ by the family of its integral closures $\overline{\bb_\bullet}$ implies the equality between the resurgence and asymptotic resurgence numbers. We call a graded family $\bb_\bullet$ of ideals a \emph{$\bb$-equivalent} family, for some ideal $\bb$, if there exists a positive integer $k$ such that, for all $i \ge 1$,
$$\bb_{i+k} \subseteq \bb^i \subseteq \bb_i.$$
Examples of $\bb$-equivalent families include that of ordinary powers or their integral closures of a given ideal in an analytically unramified ring. We generalize \cite[Corollary 4.14]{DFMS} to the following statement.

\medskip

\noindent\textbf{Theorem \ref{thm.b1equivalent}.} Let $S$ be a domain that belongs to one of the following types:
\begin{enumerate}
	\item complete local Noetherian ring,
	\item finitely generated over a field or over $\ZZ$,
	\item or, more generally, finitely generated over a Noetherian integrally closed domain $R$ satisfying the property that every finitely generated $R$-algebra has a module-finite integral closure.
\end{enumerate}
Let $\aa_\bullet$ be a filtration and let $\bb_\bullet$ be a graded family of nonzero ideals in $S$. Suppose that $\bb_\bullet$ is $\bb$-equivalent for some ideal $\bb \subseteq S$. Then,
$$\rhat(\aa_\bullet,\bb_\bullet) = \rhat(\aa_\bullet,\overline{\bb_\bullet}) = \rho(\aa_\bullet,\overline{\bb_\bullet}).$$

\medskip

When the resurgence and asymptotic resurgence numbers cannot be computed explicitly, it is desirable to know whether these numbers can be realized as actual limits of well-constructed sequences. To answer this question, we define the following sequences.
For $s \ge 1$ and a valuation $v$ of $K = \text{QF}(S)$, when $S$ is a domain, set
$$\beta_s(\aa_\bullet, \bb_\bullet) := \inf \{d \mid \aa_s \not\subseteq \bb_d\} \text{ and } \beta_s^v(\aa_\bullet, \bb_\bullet) := \inf \{d \mid v(\aa_s) < v(\bb_d)\}.$$
Also, for $n \ge 1$, define
$$\rho^n(\aa_\bullet, \bb_\bullet) := \sup \left\{ \dfrac{s}{\beta_s(\aa_\bullet, \bb_\bullet)} ~\Big|~ \beta_s(\aa_\bullet, \bb_\bullet) < \infty \text{ and } s \ge n\right\}.$$
We prove the following theorems.

\medskip

\noindent\textbf{Theorem \ref{thm.limbetarho}.}
Let $S$ be a domain, let $\aa_\bullet$ be a graded family of ideals, and let $\bb_\bullet$ be a filtration of ideals in $S$. For $n \ge 1$, set $\beta_n=\beta_n(\aa_\bullet,\bb_\bullet)$, $\overline{\beta_n}=\beta_n(\aa_\bullet,\overline{\bb_\bullet})$, and for any valuation $v$ of $K$, set $\beta_n^v=\beta_n^v(\aa_\bullet,\bb_\bullet)$. Suppose that $\R^{[k]}(\bb_\bullet)$ is a standard graded $S$-algebra and $\R(\overline{\bb_k^{\bullet}})$ is a finitely generated $\R(\bb_k^{\bullet})$-module, for some $k \in \NN$. Then, there exists a valuation $v_0$ (which can be chosen as a Rees valuation of $\bb_k$) such that
$$\frac{1}{\widehat{\rho}(\aa_\bullet,\overline{\bb_\bullet})} = \lim_{n\rightarrow \infty} \frac{\beta_{n}}{n}  = \lim_{n\rightarrow \infty} \frac{\overline{\beta_{n}}}{n}  = \lim_{n\rightarrow \infty}  \frac{\beta_{n}^{v_0}}{n}.$$
	
\medskip

\noindent\textbf{Theorem \ref{thm.rhatrlimrho}.}
Let $S$ be a domain as in \Cref{thm.b1equivalent}. Let $\aa_\bullet$ be filtration of nonzero ideals in $S$, and $\bb_\bullet$ be a $\bb$-equivalent graded family, for some ideal $\bb \subseteq S$. Then,
$$\rhat(\aa_\bullet, \overline{\bb_\bullet})= \lim_{n \to \infty} \rho^n(\aa_\bullet, \overline{\bb_\bullet})=\rhat(\aa_\bullet, \bb_\bullet) = \lim_{n \to \infty} \rho^n(\aa_\bullet, \bb_\bullet).$$
\medskip

To characterize pairs of graded families $(\aa_\bullet,\bb_\bullet)$ of ideals for which $\rho(\aa_\bullet,\bb_\bullet) < \infty$, when $S$ is an arbitrary Noetherian ring, we make use of the topology that  a filtration of ideals defines. The topology $\tau_\aa$ given by a filtration $\aa_\bullet$ is said to be \emph{linearly finer} than the topology $\tau_\bb$ given by a filtration $\bb_\bullet$ if there exists a \emph{linear} function $f \in \ZZ_{\ge 0}[x]$ such that for every $i \ge 1$, $\aa_{f(i)} \subseteq \bb_i$. Our next result is stated as follows.

\medskip

\noindent\textbf{Theorem \ref{linearly_smaller}.} Let $\aa_\bullet $ and $\bb_\bullet $ be filtration of ideals in $S$. Then, $\tau_{\aa}$ is linearly finer than $\tau_{\bb}$ if and only if  $\rho(\aa_\bullet,\bb_\bullet) <\infty.$

\medskip

The rationality of resurgence and asymptotic resurgence numbers have also been addressed by many authors. We generalize \cite[Theorem 3.7]{DD2020} and prove the following results.

\medskip

\noindent\textbf{Corollaries \ref{cor.rationalAR} and \ref{cor.rationalR}.} Let $\aa_\bullet$ and $\bb_\bullet$ be filtration of nonzero ideals in $S$.
\begin{enumerate}
	\item Suppose that $\R^{[k]}(\aa_\bullet)$ and  $\R^{[\ell]}(\bb_\bullet)$ are standard graded $S$-algebras for some $k$ and $\ell$, and that $\R(\overline{\bb_\bullet})$ is a finitely generated $\R(\bb_\bullet)$-module. Then, $\rhat(\aa_\bullet, \bb_\bullet) = \rhat(\aa_\bullet, \overline{\bb_\bullet})$ is either infinity or a rational number.
	\item Suppose, in addition, that $S$ is an analytically unramified local ring and $\bb_\bullet$ is $\bb$-equivalent, for some ideal $\bb \subseteq S$. Then, $\rho(\aa_\bullet, \bb_\bullet)$ is either infinity or a rational number.
\end{enumerate}

\medskip

Finally, we obtain a criterion for the rationality of the resurgence number, which generalizes \cite[Proposition 2.2]{DD2020}.

\medskip

\noindent\textbf{\Cref{cor.rhoRatLim}.} Let $S$, $\aa_\bullet$ and $\bb_\bullet$ be as in \Cref{thm.b1equivalent}. If $\rhat(\aa_\bullet, \bb_\bullet) \not= \rho(\aa_\bullet, \bb_\bullet)$ then $\rho(\aa_\bullet, \bb_\bullet)$ is a rational number.

\medskip

The paper is outlined as follows. In the next section, we provide a lower bound for the asymptotic resurgence $\rhat(\aa_\bullet, \overline{\bb_\bullet})$ and look at the question of when this number can be computed via Rees valuations of certain member of the family $\bb_\bullet$. Section \ref{sec.ResAndIntCl} examines under which conditions the equality $\rhat(\aa_\bullet, \bb_\bullet) = \rhat(\aa_\bullet, \overline{\bb_\bullet})$ holds. Section \ref{sec.equiv} is devoted to the case when $\bb_\bullet$ is $\bb$-equivalent for an ideal $\bb \subseteq S$. We show that, in this case, the equality $\rhat(\aa_\bullet, \bb_\bullet) = \rhat(\aa_\bullet, \overline{\bb_\bullet})$ holds and this common value is also equal to $\rho(\aa_\bullet, \overline{\bb_\bullet})$. Particularly, replacing $\bb_\bullet$ by its integral closure results in the equality between resurgence and asymptotic resurgence numbers. In Section \ref{sec.compute}, we define new sequences whose limits realize the asymptotic resurgence $\rhat(\aa_\bullet, \overline{\bb_\bullet})$. We also consider basic properties of another version of the resurgence, namely, $\rho^{\lim}(\aa_\bullet, \bb_\bullet)$. Section \ref{sec.finite} investigates the finiteness of rationality of the resurgence and asymptotic resurgence numbers.

We refer the reader to \cite{SH2006} for unexplained terminology. Throughout the paper, $S$ will denote a Noetherian commutative ring.

A collection $\aa_\bullet = \{\aa_i\}_{i \ge 1}$ of ideals in $S$ is called a \emph{graded family} if $\aa_p  \aa_q \subseteq \aa_{p+q}$ for all $p,q \ge 1$. A graded family $\aa_\bullet$ is called a \emph{filtration} if $\aa_p \supseteq \aa_{p+1}$ for all $p \ge 1$. Typical examples of filtration of ideals are those of symbolic powers, the integral closures of powers, and ordinary powers of an ideal $I \subseteq S$.
For simplicity of notations, we shall use $I^{(\bullet)}$, $\overline{I^\bullet}$ and $I^\bullet$ to denote the filtration $\{I^{(i)}\}_{i \ge 1}$, $\{\overline{I^i}\}_{i \ge 1}$ and $\{I^i\}_{i \ge 1}$, respectively. We sometimes consider a family $\aa_\bullet = \{\aa_i\}_{i \ge 0}$, with the convention that $\aa_0 = S$.

\begin{acknowledgment} The authors would like to thank A.V. Jayanthan, who contributes and has many illuminating discussions with us at the preliminary stage of this paper.
	
	The first author (THH) is partially supported by Louisiana Board of Regents and the Simons Foundation. The second author (AK) is partially supported by Sciences and Engineering Research Board, India under the National Postdoctoral Fellowship (PDF/2020/001436). The third author (HDN) acknowledges the generous support from the Vietnam Academy of Science and Technology (grants CSCL01.01/22-23 and NCXS02.01/22-23).
	
	Part of this work was done when the authors were visiting the Vietnam Institute of Advanced Study in Mathematics (VIASM). We thank VIASM for its hospitality.
\end{acknowledgment}

%%%%%%%%%%%%%%%%%%%%%%%%%%%%%%%%%
%%%%%%%%%%%%%%%%%%%%%%%%%%%%%%%%%

\section{Resurgence numbers via Rees valuations} \label{sec.viaIntegralClosure}

The aim of this section is to investigate the asymptotic resurgence number of the pair $(\aa_\bullet, \overline{\bb_\bullet})$, where $\overline{\bb_\bullet} = \{\overline{\bb_i}\}_{i \ge 1}$. In \cite[Theorem 4.10]{DFMS}, it was shown that, for an ideal $I \subseteq S$, the asymptotic resurgence number $\rhat(I^{(\bullet)}, I^\bullet)$ of the families of symbolic and ordinary powers of $I$ can be computed via Rees valuations of $I$. As we shall see, this is no longer the case for the asymptotic resurgence of an arbitrary pair $(\aa_\bullet, \bb_\bullet)$ of graded families of ideals in $S$.

We will show that the asymptotic resurgence $\rhat(\aa_\bullet, \overline{\bb_\bullet})$ can be computed via Rees valuations if $\R(\bb_\bullet)$ is Noetherian; see \Cref{thm.ReesVal}. Coupled with results in Section \ref{sec.ResAndIntCl}, this shall imply that the asymptotic resurgence $\rhat(\aa_\bullet, \bb_\bullet)$ can be computed via Rees valuations when $\R(\bb_\bullet)$ is Noetherian and $\R(\overline{\bb_\bullet})$ is a finitely generated module over $\R(\bb_\bullet)$. Furthermore, we shall give a lower bound for $\rhat(\aa_\bullet, \overline{\bb_\bullet})$ in terms of \emph{skew Waldschmidt constants} of $\aa_\bullet$ and $\bb_\bullet$ with respect to a valuation of the quotient field of $S$; see \Cref{low_bou}.

We will begin by recalling two omnipresent constructions of the Rees algebra and its Veronese subalgebras.

\begin{definition}
	Let $\aa_\bullet = \{\aa_i\}_{i \ge 1}$ be a graded family of ideals in $S$ (with the convention that $\aa_0 = S$). 
\begin{enumerate} 
\item The \emph{Rees algebra} of $\aa_\bullet$ is defined to be
	$$\R(\aa_\bullet) = \bigoplus_{n \ge 0} \aa_n t^n \subseteq S[t].$$
\item For $k \in \NN$, the \textit{$k$-th Veronese subalgebra} of $\R(\aa_\bullet)$ is given by
    $$\R^{[k]}(\aa_\bullet) = \bigoplus_{n \ge 0} \aa_{kn} t^{kn} \subseteq S[t].$$
\end{enumerate}
\end{definition}

In the general, the Rees algebra of a graded family of ideals is not necessarily Noetherian, even when the graded family is that of symbolic powers of an ideal in $S$ (cf. \cite{C1991, H1982, R1985}).
It is a folklore result that a graded ring $\R = \bigoplus_{n \ge 0} R_n$ is Noetherian if and only if $R_0$ is a Noetherian ring and $\R$ is a finitely generated algebra over $R_0$.
We shall also say that the graded family $\aa_\bullet$ is \emph{Noetherian} if its Rees algebra is. It is well know that, see \cite[Remark 2.4]{Rat} and \cite[Proposition 2.1]{Sch88}, that if $\aa_\bullet$ is Noetherian graded family, then $\R^{[k]}(\aa_\bullet)$ is a standard graded $S$-algebra for some $k$.

For graded families $\aa_\bullet = \{\aa_i\}_{i \ge 1}$ and $\bb_\bullet = \{\bb_i\}_{i \ge 1}$ of ideals in $S$, we write $\aa_\bullet \le \bb_\bullet$ if $\aa_i \subseteq \bb_i$ for all $i \ge 1$. The following lemma follows directly from the definition.

\begin{lemma}
	\label{lem.estimate}
	Let $\aa_\bullet, \bb_\bullet, \aa_\bullet'$ and $\bb_\bullet'$ be graded families of ideals in $S$.
	\begin{enumerate}
		\item If $\bb_\bullet \le \bb_\bullet'$, then
		$$\rho(\aa_\bullet, \bb_\bullet) \ge \rho(\aa_\bullet, \bb_\bullet') \text{ and } \rhat(\aa_\bullet, \bb_\bullet) \ge \rhat(\aa_\bullet, \bb_\bullet').$$
		\item If $\aa_\bullet \le \aa_\bullet'$, then
		$$\rho(\aa_\bullet, \bb_\bullet) \le \rho(\aa_\bullet', \bb_\bullet) \text{ and } \rhat(\aa_\bullet, \bb_\bullet) \le \rhat(\aa_\bullet', \bb_\bullet).$$
	\end{enumerate}
\end{lemma}

As an immediate consequence of \Cref{lem.estimate}, we obtain the following result.

\begin{corollary}
	\label{cor.Int}
	Let $\aa_\bullet$ and $\bb_\bullet$ be graded families of ideals in $S$. Then,
	\begin{enumerate}
		\item $\rho(\aa_\bullet, \bb_\bullet) \ge \rho(\aa_\bullet, \overline{\bb_\bullet})$ and $\rhat(\aa_\bullet, \bb_\bullet) \ge \rhat(\aa_\bullet, \overline{\bb_\bullet})$;
		\item $\rho(\aa_\bullet, \bb_\bullet) \le \rho(\overline{\aa_\bullet}, \bb_\bullet)$ and $\rhat(\aa_\bullet, \bb_\bullet) \le \rhat(\overline{\aa_\bullet}, \bb_\bullet)$.
	\end{enumerate}
\end{corollary}

Throughout this section, $S$ is assumed to be a Noetherian domain. Let $K$ denote the quotient field of $S$.

\begin{definition}
A \emph{(discrete) valuation} of $K$ is a function $v : K \setminus \{0\} \rightarrow \ZZ$ satisfying the following properties for all $x, y \in K \setminus \{0\}$:
	\begin{enumerate}
		\item[(i)] $v(xy) = v(x) + v(y)$, and
		\item[(ii)] $v(x+y) \ge \min \{v(x), v(y)\}.$
	\end{enumerate}
\end{definition}

We shall omit the word ``discrete'' from ``discrete valuation'' in this paper, as no confusion will be resulted.
A valuation $v$ of $K$ is said to be \emph{supported} on $S$ if $v(x) \ge 0$ for all $x \in S \setminus \{0\}$.% or, equivalently, $S \subseteq R_v$.

Recall that for a valuation $v$ of $K$ and an ideal $I \subseteq S$, $v(I)$ denotes $\min \{v(x) ~\big|~ x \in I \setminus \{0\}\}.$ For a graded family $\aa_\bullet$ of ideals in $S$, the \emph{skew Waldschmidt constant} of $\aa_\bullet$ associated to $v$ is
$$\vhat(\aa_\bullet) = \lim_{n \rightarrow \infty} \dfrac{v(\aa_n)}{n} = \inf_{n \in \NN} \dfrac{v(\aa_n)}{n}.$$

Before stating our first result, we recall the definition of Rees valuations following \cite{SH2006}.

\begin{definition}[{\cite[Definition 10.1.1]{SH2006}}] \label{def.ReesVal}
	Let $I \subseteq S$ be an ideal. There exist finitely many valuation rings $V_1, \dots, V_r$ of $K$ satisfying the following properties:
	\begin{enumerate}
		\item[(a)] for each $i = 1, \dots, r$, $S \subseteq V_i \subseteq K$;
		\item[(b)] let $\phi_i : S \rightarrow V_i$ be the natural ring homomorphism;
		\item[(c)] for all $n \in \NN$, $\overline{I^n} = \bigcap_{i=1}^r \phi_i^{-1}(\phi_i(I^n)V_i)$; and
		\item[(d)] the set $\{V_1, \dots, V_r\}$ satisfying (c) is minimal possible.
	\end{enumerate}
	The valuation rings $V_1, \dots, V_r$ are called the \emph{Rees valuation rings} of $I$, and their corresponding valuations are called the \emph{Rees valuations} of $I$. We denote the set of Rees valuations of $I$ by $\RV(I)$. It is known that $\RV(I)$ is a finite set.
\end{definition}

Our first main result is a simple bound for the asymptotic resurgence $\rhat(\aa_\bullet, \overline{\bb_\bullet})$, which generalizes \cite[Theorem 1.2.1]{BH} from Waldschmidt constant to the skew Waldschmidt constant associated to any valuation of $K$ that is supported on $S$.

\begin{theorem} \label{low_bou}
	Let $\aa_\bullet$ and $\bb_\bullet$ be graded families of nonzero ideals in $S$. Let $v$ be a valuation of $K$, that is supported on $S$, such that $\vhat(\bb_\bullet)>0$. Then, $$\dfrac{\vhat(\bb_\bullet)}{\vhat(\aa_\bullet)} \leq \rhat(\aa_\bullet,\overline{\bb_\bullet}).$$  Moreover, if $\vhat(\aa_\bullet) =0$, then $\rhat(\aa_\bullet,\overline{\bb_\bullet})=\infty.$
\end{theorem}

\begin{proof}  Consider arbitrary $s,r \in \mathbb{N}$ with $\vhat(\aa_\bullet) < \vhat(\bb_\bullet) \dfrac{r}{s}.$ Choose $ \epsilon  >0$ such that $\vhat(\aa_\bullet) + \epsilon < \vhat(\bb_\bullet) \dfrac{r}{s}$. Note that $\left\{ \dfrac{v(\aa_{st})}{st} \right\}_{t \geq 1}$ is a subsequence of $\left\{ \dfrac{v(\aa_{n})}{n} \right\}_{n \geq 1}$. Therefore, $\lim_{t \to \infty} \dfrac{v(\aa_{st})}{st} =\vhat(\aa_\bullet)$. Thus, there exists $t_0 \in \mathbb{N}$ such that for all $t \geq t_0$,  $$\dfrac{v(\aa_{st})}{st} \leq \vhat(\aa_\bullet) + \epsilon < \vhat(\bb_\bullet) \dfrac{r}{s}.$$
	Therefore, for $t \geq t_0$, we have $${v(\aa_{st})} < rt\vhat(\bb_\bullet) \leq rt \dfrac{v(\bb_{rt})}{rt}=v(\bb_{rt}).$$
	It follows from \cite[Theorem 6.8.3]{SH2006} that $\aa_{st} \not\subseteq \overline{\bb_{rt}}$  for all $t \geq t_0$. It follows that
	$$\left\{ \dfrac{s}{r} ~\Big|~ s,r \in \mathbb{N} \text{ and } \vhat(\aa_\bullet) < \vhat(\bb_\bullet) \dfrac{r}{s}\right\} $$
	is a subset of
	$$\left\{ \dfrac{s}{r} ~\Big|~ s,r \in \mathbb{N} \text{ and } \aa_{st} \not\subseteq \overline{\bb_{rt}} \text{ for } t \gg 1\right\},$$
	and hence, $\dfrac{\vhat(\bb_\bullet)}{\vhat(\aa_\bullet)} \leq \rhat(\aa_\bullet,\overline{\bb_\bullet}).$ The second assertion follows from the fact that if $\vhat(\aa_\bullet)=0$, then $\mathbb{N} \subseteq \left\{ \dfrac{s}{r} ~\Big|~ s,r \in \mathbb{N} \text{ and } \aa_{st} \not\subseteq \overline{\bb_{rt}} \text{ for } t \gg 1\right\}$. This completes the theorem.
\end{proof}

The bound for $\rhat(\aa_\bullet, \overline{\bb_\bullet})$ in \Cref{low_bou} is sharp, as seen in the following example.

\begin{example}
Let $\aa_\bullet$ be a filtration of nonzero proper ideals in $S$ and let $\bb$ be an ideal of $S$. Take $\bb_i=\bb^i$ for all $i \ge 1$.  We first claim that $\vhat(\aa_\bullet)>0$ for all $v \in \text{RV}(\bb)$ if and only if $\rho(\aa_\bullet, \overline{\bb_\bullet})< \infty.$ Suppose that $\rho(\aa_\bullet, \overline{\bb_\bullet}) < \infty.$ Then, there exists a positive integer $n$ such that $\aa_{ni} \subseteq \overline{\bb^i}$ for all $i$.  Therefore, for any $v \in \text{RV}(\bb)$ and  for all $i$, $v(\aa_{ni}) \ge iv(\bb)$ which implies that $\dfrac{v(\aa_{ni})}{ni} \ge \dfrac{v(\bb)}{n}$. Consequently, $\vhat(\aa_\bullet)>0$ for all $v \in \text{RV}(\bb)$.

Conversely, we assume that $\vhat(\aa_\bullet)>0$ for all $v \in \text{RV}(\bb)$. For each $v \in \text{RV}(\bb)$, choose $n_v \in \NN$ such that $\vhat(\aa_\bullet) \ge \dfrac{v(\bb)}{n_v}.$ Take $n=\max\{n_v \mid v\in \text{RV}(\bb)\}.$ Thus,  for all $v \in \text{RV}(\bb)$, $\vhat(\aa_\bullet) \ge \dfrac{v(\bb)}{n}$ which implies that $\dfrac{v(\aa_{ni})}{ni} \ge \vhat(\aa_\bullet) \ge \dfrac{v(\bb)}{n}$. Consequently, $v(\aa_{ni})\ge v(\bb^i)$ for all $i$ and for all $v\in \text{RV}(\bb).$ Now, by \cite[Theorem 6.8.3]{SH2006}, we get $\aa_{ni}\subseteq \overline{\bb^i}$ for all $i$. Hence, $\rho(\aa_\bullet,\overline{\bb_\bullet})<\infty.$ The claim follows.

To continue, let $s,r  \in \NN$ be such that $\dfrac{s}{r} \ge \max\limits_{v \in \text{RV}(\bb)} \left\{\dfrac{v(\bb)}{\vhat(\aa_\bullet)}\right\}.$ Then for all $v \in \text{RV}(\bb)$, we have $v(\aa_s) \ge s \vhat(\aa_\bullet) \ge v(\bb^r)$. This and \cite[Theorem 6.8.3]{SH2006} imply that $\aa_s \subseteq \overline{\bb^r}$. Therefore, $\rho(\aa_\bullet, \overline{\bb_\bullet}) \le \max\limits_{v \in \text{RV}(\bb)} \left\{\dfrac{v(\bb)}{\vhat(\aa_\bullet)}\right\}$ and, hence,
$$\rhat(\aa_\bullet, \overline{\bb_\bullet}) = \rho(\aa_\bullet, \overline{\bb_\bullet}) = \max_{v \in \text{RV}(\bb)} \left\{\dfrac{v(\bb)}{\vhat(\aa_\bullet)}\right\}.$$
Furthermore, if $\bb$ is normal ideal, i.e., $\bb^i=\overline{\bb^i}$ for all $i$, then $\rhat(\aa_\bullet, \overline{\bb_\bullet}) = \rhat(\aa_\bullet, \bb_\bullet) =\rho(\aa_\bullet, \bb_\bullet) =\rho(\aa_\bullet,\overline{ \bb_\bullet}).$
 \end{example}

\begin{corollary}\label{cor.ReesVal}
	Let $\aa_\bullet$ and $\bb_\bullet$ be graded families of nonzero ideals  in $S$.   Then,
	$$\rhat(\aa_\bullet,\overline{\bb_\bullet}) \ge \sup_{\vhat(\bb_\bullet) > 0} \left\{ \dfrac{\vhat(\bb_\bullet)}{\vhat(\aa_\bullet)} \right\} ,$$
	where the supremum is taken over all valuations $v$ of $K$ supported on $S$ for which $\vhat(\bb_\bullet)>0$.
\end{corollary}

\begin{proof} The assertion follows immediately from \Cref{low_bou}.
\end{proof}

Note that the Waldschmidt constant of a graded family of homogeneous ideals in a graded ring  is a special case of the skew Waldschmidt constant when the valuation of an element is given by the degree of that element. As an immediate consequence of \Cref{low_bou}, we obtain the following generalization of \cite[Theorem 1.2.1]{BH}.

\begin{corollary}[{See \cite[Theorem 1.2.1]{BH}}] \label{cor.boundsW}
	Let $S$ be a polynomial ring, and let $\aa_\bullet $ and $\bb_\bullet $ be graded families of nonzero homogeneous ideals in $S$. Suppose that $\ahat(\bb_\bullet) >0$. Then, $$\dfrac{\ahat(\bb_\bullet)}{\ahat(\aa_\bullet)} \leq \rhat(\aa_\bullet,\overline{\bb_\bullet}) \leq \rhat(\aa_\bullet,\bb_\bullet).$$ Moreover, if $\ahat(\aa_\bullet) =0$, then $\rhat(\aa_\bullet, \overline{\bb_\bullet})=\rhat(\aa_\bullet,\bb_\bullet) =\infty.$
\end{corollary}

\Cref{low_bou} and \Cref{cor.boundsW} are not necessarily true without the condition that $\vhat(\bb_\bullet) > 0$, as illustrated in the following example.

\begin{example}\label{ex.low}
	\label{ex.zeroahat}
	Let $I$ be a nonzero proper normal ideal in $S$, and let $v$ be a valuation of $K$ that is supported on $S$ and $v(I) > 0$.
	\begin{enumerate}
		\item Consider the filtration $\aa_\bullet $ and $\bb_\bullet $ with
		$$\aa_i = I^i \text{ and } \bb_i = I \text{ for all } i \ge 1.$$
		Clearly, $\vhat(\aa_\bullet) = v(I)$ and $\vhat(\bb_\bullet) = 0.$ It can be seen that, in this example,
		$$\rho(\aa_\bullet,\bb_\bullet)=\rhat(\aa_\bullet,\bb_\bullet)=\rhat(\aa_\bullet, \overline{\bb_\bullet}) =\sup \varnothing =-\infty < \dfrac{\vhat(\bb_\bullet)}{\vhat(\aa_\bullet)}.$$
		\item Consider slightly modified families $\aa_\bullet$ and $\bb_\bullet $ with $\aa_i = I^i$ and
		$$\bb_i = \left\{ \begin{array}{lll} I & \text{if} & i \not= 2 \\ I^2 & \text{if} & i = 2.\end{array}\right.$$
		Then, $\vhat(\aa_\bullet) = v(I)$ and $\vhat(\bb_\bullet) = 0.$ It can be seen that, in this example,
		$$\rho(\aa_\bullet,\bb_\bullet)=\dfrac{1}{2} \text{ and } \rhat(\aa_\bullet,\bb_\bullet) = \rhat(\aa_\bullet, \overline{\bb_\bullet}) =-\infty < \dfrac{\vhat(\bb_\bullet)}{\vhat(\aa_\bullet)}.$$
		\item Consider the filtration $\aa_\bullet $ and $\bb_\bullet $ with $$\aa_i =I^i \text{ and } \bb_i = I^{\lceil{\sqrt i}\rceil} \text{ for all } i \ge 1.$$
Observe that $v(\bb_i) =\lceil \sqrt{i} \rceil  v(I)$ for all $i \ge 1$.
		Therefore, $\vhat(\aa_\bullet) = v(I)$ and $\vhat(\bb_\bullet) = 0.$

We claim that
		$$\rho(\aa_\bullet,\bb_\bullet)=\dfrac{1}{2} \text{ and } \rhat(\aa_\bullet,\bb_\bullet) = \rhat(\aa_\bullet, \overline{\bb_\bullet}) =-\infty < \dfrac{\vhat(\bb_\bullet)}{\vhat(\aa_\bullet)}.$$
Indeed, let $s,r$ be positive integers such that $r<2s$. Then, $s^2\ge r$ which implies that $s\ge \lceil \sqrt r \rceil.$ This implies that $\aa_s \subseteq \bb_r,$ and so,  $\rho(\aa_\bullet,\bb_\bullet)\le \dfrac{1}{2}.$ Furthermore, since $\aa_1 =I \not\subseteq I^2=\bb_2$, we have $\rho(\aa_\bullet,\bb_\bullet)=\dfrac{1}{2}.$

Also, if there exist positive integers $s,r$ such that $\aa_{st}\not\subseteq \bb_{rt}$ for $t \gg 1$, then $st < \lceil \sqrt{rt} \rceil$ for $t\gg 1$. Replacing $t$ by $rt^2$, we get $st<1$ for $t\gg 1$, which is a contradiction. Hence,  $\rhat(\aa_\bullet,\bb_\bullet) = -\infty.$

\item Consider the filtration $\aa_\bullet $ and $\bb_\bullet $ with $$\aa_i = \bb_i = I^{\lceil{\sqrt i}\rceil} \text{ for all } i \ge 1.$$
Observe that $v(\aa_i)=v(\bb_i) =\lceil \sqrt{i} \rceil  v(I)$ for all $i \ge 1$.
		Therefore, $\vhat(\aa_\bullet) =\vhat(\bb_\bullet) = 0.$
We claim that
		$$\rho(\aa_\bullet,\bb_\bullet)= \rhat(\aa_\bullet,\bb_\bullet)=\rhat(\aa_\bullet, \overline{\bb_\bullet})  = 1.$$
Indeed, let $s,r$ be positive integers such that $s \ge r$. Then,  $\lceil \sqrt s \rceil \ge \lceil \sqrt r \rceil$,  which implies that $\aa_s \subseteq \bb_r,$ and so,  $\rho(\aa_\bullet,\bb_\bullet)\le 1.$ Also, observe that for any $s$, $\lceil \sqrt{st} \rceil > \lceil \sqrt{(s-1)t} \rceil$ for $t\gg 1$.  Therefore,  for any $s$, $\aa_{(s-1)t}\not\subseteq \bb_{st}$ for $t \gg 1$. This implies that $\dfrac{s-1}{s} \le \rhat(\aa_\bullet, \bb_\bullet)$ for all $s$. Hence,  $\rhat(\aa_\bullet,\bb_\bullet) = \rho(\aa_\bullet, \bb_\bullet)=1.$
	\end{enumerate}
\end{example}

The next lemma examines the condition that $\vhat(\bb_\bullet) > 0$ in \Cref{low_bou}.

\begin{lemma}\label{supported}
	Let $\bb_\bullet$ be a graded family of ideals in $S$ such that $\R^{[k]}(\bb_\bullet)$ is a standard graded $S$-algebra. Then,  $\vhat(\bb_\bullet) =\dfrac{v(\bb_k)}{k}$ for any valuation $v$ of $K$. In particular, $\vhat(\bb_\bullet) >0$ if and only if $v(\bb_k) > 0$.
\end{lemma}

\begin{proof}
	Since $\R^{[k]}(\bb_\bullet)$ is a standard graded $S$-algebra, $\bb_{ks}=\bb_k^s$ for all $s \ge 1$.  Therefore, $v(\bb_{ks}) =s v(\bb_k)$ for all $s \in \mathbb{N}.$ Since $\left\{\dfrac{v(\bb_{ks})}{ks}\right\}$ is a subsequence of $\left\{\dfrac{v(\bb_s)}{s}\right\},$ $\vhat(\bb_\bullet) = \displaystyle \lim_{s \to \infty} \dfrac{v(\bb_{ks})}{ks}= \dfrac{v(\bb_{k})}{k}.$
\end{proof}

We are now ready to present our next main result of this section, which generalizes \cite[Theorem 4.10]{DFMS}. Observe that following \cite[Chapter 10]{SH2006}, for any ideal $I \subseteq S$ and any Rees valuation $v \in \RV(I)$ of $I$, we have that $v(I) > 0$.

\begin{theorem}
	\label{thm.ReesVal}
	Let $\aa_\bullet$ and $\bb_\bullet$ be graded families of nonzero ideals in $S$, and let $k \in \NN$. Suppose that $\R^{[k]}(\bb_\bullet)$ is a standard graded $S$-algebra. Then,
	$$\rhat(\aa_\bullet,\overline{\bb_\bullet}) =\max_{v \in \RV(\bb_k)} \left\{ \dfrac{\vhat(\bb_\bullet)}{\vhat(\aa_\bullet)} \right\} =\sup_{v(\bb_k) > 0} \left\{ \dfrac{\vhat(\bb_\bullet)}{\vhat(\aa_\bullet)} \right\} ,$$
	where the supremum is taken over all valuations of $K$ supported on $S$ that take positive values in $\bb_k$. Particularly, if $\R^{[\ell]}(\aa_\bullet)$ is a standard graded $S$-algebra for some $\ell$, then $\rhat(\aa_\bullet, \overline{\bb_\bullet})$ is either a positive rational number or infinite.
\end{theorem}

\begin{proof}  Since $\R^{[k]}(\bb_\bullet)$ is a standard graded $S$-algebra,  $\bb_{kt} = \bb_k^t$ for all $t \ge 1$. By Lemma \ref{supported}, for any valuation $v$ of $K$ supported on $S$, $\vhat(\bb_\bullet) > 0$ if and only if $v(\bb_k)>0$. Consider any valuation $v \in \RV(\bb_k)$. As remarked above, $v(\bb_k) > 0$.  By \Cref{low_bou}, we then get
	$$\max_{v \in \RV(\bb_k)} \left\{ \dfrac{\vhat(\bb_\bullet)}{\vhat(\aa_\bullet)} \right\} \le \sup\limits_{v(\bb_k) > 0}  \left\{ \dfrac{\vhat(\bb_\bullet)}{\vhat(\aa_\bullet)} \right\} \leq \rhat(\aa_\bullet,\overline{\bb_\bullet}).$$
	Thus, it suffices to prove that
	$$\rhat(\aa_\bullet,\overline{\bb_\bullet}) \leq \max_{v \in \RV(\bb_k)} \left\{ \dfrac{\vhat(\bb_\bullet)}{\vhat(\aa_\bullet)} \right\}.$$

	If for some $v \in \RV(\bb_k)$, $\vhat(\aa_\bullet) =0$, then by \Cref{low_bou},
$\rhat(\aa_\bullet,\overline{\bb_\bullet}) =\infty= \max\limits_{v \in \RV(\bb_k)} \left\{ \dfrac{\vhat(\bb_\bullet)}{\vhat(\aa_\bullet)} \right\} .$
So, we assume that 	$\vhat(\aa_\bullet) > 0$ for all $v \in \RV(\bb_k)$.
	Suppose, by contradiction, that
	$$\rhat(\aa_\bullet,\overline{\bb_\bullet}) > \max_{v \in \RV(\bb_k)} \left\{ \dfrac{\vhat(\bb_\bullet)}{\vhat(\aa_\bullet)} \right\}.$$
 By definition, there exist $s,r \in \NN$ such that
	$\rhat(\aa_\bullet, \overline{\bb_\bullet}) \ge \dfrac{s}{r} > \max\limits_{v \in \RV(\bb_k)} \left\{ \dfrac{\vhat(\bb_\bullet)}{\vhat(\aa_\bullet)}\right\},$
	and $\aa_{st} \not\subseteq \overline{\bb_{rt}}$ for $t \gg 1$.
Since $\left\{\dfrac{v(\bb_{kt})}{kt}\right\}$ is a subsequence of $\left\{\dfrac{v(\bb_n)}{n}\right\},$ we get $\vhat(\bb_\bullet) = \displaystyle \lim_{t \to \infty} \dfrac{v(\bb_{kt})}{kt}= \dfrac{v(\bb_{k})}{k}=\dfrac{v(\bb_{kt})}{kt}$ for all $t \geq 1$. Furthermore, $\RV(\bb_{kt}) = \RV(\bb_k)$ for any $t \ge 1$ (cf. \cite[Exercise 10.1]{SH2006}).
 Now, for $p \gg 1$, we have $\aa_{sk^p} \not\subseteq \overline{\bb_{rk^p}}.$
	This, by \cite[Theorem 6.8.3 and Chapter 10]{SH2006}, implies that for some $w \in \RV(\bb_{rk^p}) = \RV(\bb_k)$ ($w$ depends on $p$), we have $w(\aa_{sk^p}) <w({\bb_{rk^p}}).$
	Therefore,
	\begin{align*}
		\dfrac{s}{r} & > \dfrac{\widehat{w}(\bb_\bullet)}{\widehat{w}(\aa_\bullet)} = \dfrac{w(\bb_{rk^p})}{rk^p \widehat{w}(\aa_\bullet)}
		\geq \dfrac{w(\bb_{rk^p})}{rk^p \dfrac{w(\aa_{sk^p})}{sk^p}} \\
		& = \dfrac{s}{r}  \dfrac{w(\bb_{rk^p})}{w(\aa_{sk^p})} > \dfrac{s}{r},
	\end{align*}
	which is a contradiction. Thus, $\rhat(\aa_\bullet,\overline{\bb_\bullet}) \leq \max\limits_{v \in \RV(\bb_k)} \left\{ \dfrac{\vhat(\bb_\bullet)}{\vhat(\aa_\bullet)} \right\}.$
	Hence, the first assertion follows.

Since $\R^{[\ell]}(\aa_\bullet)$ and $\R^{[k]}(\bb_\bullet)$ are standard graded $S$-algebras, by Lemma \ref{supported}, $\vhat(\aa_\bullet)$ is a non-negative rational number and  $\vhat(\bb_\bullet)$ is a positive rational numbers for any $v \in \RV(\bb_k)$. The last statement follows from the fact that there are finitely many Rees valuations for any ideal $\bb_k$.
\end{proof}

As an immediate consequence of \Cref{thm.ReesVal}, we obtain the following result.

\begin{corollary}
	\label{cor.NoethRees}
	Let $\aa_\bullet$ and $\bb_\bullet$ be graded families of nonzero ideals in $S$. Suppose that the Rees algebra $\R(\bb_\bullet)$ is Noetherian. Then, there exists an integer $k$ such that
	$$\rhat(\aa_\bullet,\overline{\bb_\bullet}) =\max_{v \in \RV(\bb_k)} \left\{ \dfrac{\vhat(\bb_\bullet)}{\vhat(\aa_\bullet)} \right\} =\sup_{v(\bb_k) > 0} \left\{ \dfrac{\vhat(\bb_\bullet)}{\vhat(\aa_\bullet)} \right\} ,$$
	where the supremum is taken over all valuations of $K$ supported on $S$ that take positive values in $\bb_k$. Particularly, if $\R^{[\ell]}(\aa_\bullet)$ is a standard graded $S$-algebra for some $\ell$, then $\rhat(\aa_\bullet, \overline{\bb_\bullet})$ is either a positive rational number or infinite.
\end{corollary}

Without the condition that the $k$-th Veronese subring $\R^{[k]}(\bb_\bullet)$ is a standard graded $S$-algebra, the conclusion of \Cref{thm.ReesVal} may not hold and the asymptotic resurgence number $\rhat(\aa_\bullet, \bb_\bullet)$ could be an irrational number, as demonstrated in the next two examples.

\begin{example} \label{ex.irrational}
 Let $I$ be a nonzero proper ideal in $S$. Fix $\alpha, \beta \in \mathbb{R}_{>0}$. Consider the families $\aa_\bullet = \{\aa_i\}_{i \ge 1}$ and $\bb_\bullet= \{\bb_i\}_{i \ge 1}$ given by
 $$\aa_i=I^{\lceil \alpha i \rceil} \text{ and } \bb_i =I^{\lceil \beta i \rceil}.$$
 It is easy to verify that $\aa_\bullet$ and $\bb_\bullet$ are filtration of ideals in $S$.

 Let $v \in \text{RV}(I)$ be any Rees valuation of $I$. It can be seen that $\vhat(\aa_\bullet)=\alpha v(I)$ and $\vhat(\bb_\bullet) = \beta v(I)$. Therefore, by \Cref{low_bou}, $\rhat(\aa_\bullet, \overline{\bb_\bullet}) \ge \dfrac{\vhat(\bb_\bullet)}{\vhat(\aa_\bullet)} = \dfrac{\beta}{\alpha}.$

On the other hand, for any $s,r \in \NN$ such that $\dfrac{s}{r} \ge \dfrac{\beta}{\alpha}$, we have $\alpha s \ge \beta r$. This implies that $\lceil \alpha s \rceil \ge \lceil \beta r \rceil$.  It then follows that $\aa_s \subseteq \bb_r.$ Therefore, $\rho(\aa_\bullet, \bb_\bullet) \le \dfrac{\beta}{\alpha}$. Hence,
$$\rhat(\aa_\bullet, \overline{\bb_\bullet}) = \rhat(\aa_\bullet, \bb_\bullet) =\rho(\aa_\bullet, \overline{\bb_\bullet})=\rho(\aa_\bullet, \bb_\bullet)=\dfrac{\beta}{\alpha}.$$

It follows from \cite[Proposition 2.1]{Sch88} that, in this example, $\R^{[k]}(\aa_\bullet)$ is a standard graded $S$-algebra only if $\alpha$ is rational. Particularly, if we choose $\alpha$ or $\beta$ to be irrational then the hypothesis of \Cref{thm.ReesVal} is not satisfied, and the resurgence and asymptotic resurgence numbers are both irrational.
 \end{example}

\begin{example} \label{ex.rhatNOTnoeth}
	Let $R$ be a Noetherian domain and let $S=R[x]$. Let $I$ be a nonzero proper ideal of $R$. Let $\aa_{\bullet}$ and $\bb_{\bullet} $ be graded families given by $$\aa_i = xI^i \text{ and } \bb_i=x^2I^i \text{ for all } i \ge 1.$$
	Direct computation shows that $\overline{\bb_i}=x^2\overline{I^i}$ for all $i\ge 1$.
	
	We claim that $\rhat(\aa_{\bullet},\overline{\bb_{\bullet}}) = \infty$.
	Indeed, for any $s,r,t \in \NN$, it can be seen that $\aa_{st} =xI^{st} \not\subseteq \overline{\bb_{rt}}=x^2\overline{I^{rt}}$ which implies that $ \dfrac{s}{r} \leq \rhat(\aa_{\bullet},\overline{\bb_{\bullet}})$. This inequality holds for any $s, r \in \NN$, so we have $\rhat(\aa_{\bullet},\overline{\bb_{\bullet}}) = \infty.$
	
	Let $v$ be a valuation of $K$ supported on $S$. Then $v(\aa_i) = v(x)+v(I^i) = v(x) +iv(I)$ and $v(\overline{\bb_i}) =2 v(x) +i v(I)$ for all $i \ge 1$.  Therefore, $\vhat(\aa_\bullet) =\vhat(\bb_\bullet) =v(I).$ Thus, we have   $$ \sup_{\vhat(\bb_\bullet) > 0} \left\{ \dfrac{\vhat(\bb_\bullet)}{\vhat(\aa_\bullet)} \right\} =1<\rhat(\aa_\bullet,\overline{\bb_\bullet}).$$
\end{example}

\begin{remark}
Let $\aa_\bullet$ and $\bb_\bullet$ be graded families of nonzero ideals in $S$.
\begin{enumerate}
	\item Suppose that $\R^{[k]}(\bb_\bullet)$ is a standard graded $S$-algebra. Then, $\rhat(\aa_\bullet,\overline{\bb_\bullet})<\infty$ if and only if $\vhat(\aa_\bullet) >0$ for all  $v \in \RV(\bb_k)$. This follows directly from Theorem \ref{thm.ReesVal} and the fact that $\RV(\bb_k)$ is a finite set.
	\item Suppose that $\rhat(\aa_\bullet, \overline{\bb_\bullet}) = - \infty$. Then, by \Cref{low_bou}, for any valuation $v$ of $K$ supported on $S$, $\vhat(\bb_\bullet)=0$. Therefore, the conclusion of Theorem \ref{thm.ReesVal} is still valid as $\sup\limits_{\vhat(\bb_\bullet) > 0} \left\{ \dfrac{\vhat(\bb_\bullet)}{\vhat(\aa_\bullet)} \right\}=-\infty$ in this case. We shall see this scenario in the following example.
\end{enumerate}
\end{remark}

\begin{example}
Let $I$ be a nonzero proper normal ideal in $S$. Consider  $\aa_\bullet$ and $\bb_\bullet $ with
$$\aa_i =I^i \text{ and } \bb_i = I^{\lceil{\sqrt i}\rceil} \text{ for all } i \ge 1.$$
As, we have seen in \Cref{ex.low}.(3) that for any valuation $v$ of $K$ supported on $S$,  $\vhat(\aa_\bullet) = v(I)$ and $\vhat(\bb_\bullet) = 0.$ Particularly, $$\sup_{\vhat(\bb_\bullet) > 0} \left\{ \dfrac{\vhat(\bb_\bullet)}{\vhat(\aa_\bullet)} \right\} = -\infty.$$ Also, $\rhat(\aa_\bullet,\overline{\bb_\bullet}) = -\infty.$
Hence, $\rhat(\aa_\bullet,\overline{\bb_\bullet}) = \sup\limits_{\vhat(\bb_\bullet) > 0} \left\{ \dfrac{\vhat(\bb_\bullet)}{\vhat(\aa_\bullet)} \right\}$.
\end{example}

If we assume that $\bb_\bullet$ is a filtration and $\R^{[k]}(\bb_\bullet)$ is a standard graded $S$-algebra for some $k$, then the conclusion of \Cref{thm.ReesVal} can be slightly modified to consider the supremum over those valuations which takes positive values on $\bb_1$ (instead of $\bb_k$).

\begin{corollary}
	Let $\aa_\bullet$ be a graded family and let $\bb_\bullet$ be a filtration of nonzero ideals in $S$.  Suppose that $\R^{[k]}(\bb_\bullet)$ is a standard graded $S$-algebra. Then,
	$$\rhat(\aa_\bullet,\overline{\bb_\bullet}) =\sup_{v(\bb_1) > 0} \left\{ \dfrac{\vhat(\bb_\bullet)}{\vhat(\aa_\bullet)} \right\} ,$$
	where the supremum is taken over valuations of $K$ supported on $S$ that take positive values in $\bb_1$.
\end{corollary}

\begin{proof} Since $\bb_\bullet$ is a filtration of ideals in $S$, we have $\bb_1^k \subseteq \bb_k \subseteq \bb_1$. Therefore, for any valuation $v$ of $K$, we have $ v(\bb_1) \le v(\bb_k) \leq kv(\bb_1)$. This implies that $v(\bb_1)>0$ if and only if $v(\bb_k)>0$. Now, the assertion follows from Theorem \ref{thm.ReesVal}.
\end{proof}

\begin{question} \label{ques-rees}
For which graded families $\aa_\bullet$ and $\bb_\bullet$ of ideals, does the following equality hold
$$\rhat(\aa_\bullet,\overline{\bb_\bullet}) =\max_{v \in \RV(\bb_1)} \left\{ \dfrac{\vhat(\bb_\bullet)}{\vhat(\aa_\bullet)} \right\}?$$
\end{question}

%%%%%%%%%%%%%%%%%

\begin{remark}
	\label{rmk.IntOfA}
	With the same line of arguments, one may obtain some similar results, but not all, when $(\aa_\bullet, \overline{\bb_\bullet})$ is replaced by $(\overline{\aa_\bullet}, \bb_\bullet)$. We leave this to the interested reader.
\end{remark}

%%%%%%%%%%%%%%%%%%%%%%%%%%%%%%%%%%

\section{Resurgence numbers and integral closures} \label{sec.ResAndIntCl}

This section is devoted to the study of how asymptotic resurgence numbers behave when the family $\bb_\bullet = \{\bb_i\}_{i \ge 1}$ is replaced by $\overline{\bb_\bullet} = \{\overline{\bb_i}\}_{i \ge 1}$. In \cite[Proposition 4.2 and Corollary 4.14]{DFMS}, it was shown that for an ideal $I$ in a polynomial ring $S$,
$$\rhat(I^{(\bullet)}, I^\bullet) = \rhat(I^{(\bullet)}, \overline{I^\bullet}) = \rho(I^{(\bullet)}, \overline{I^\bullet}).$$
The situation for resurgence and asymptotic resurgence numbers of pairs of graded families of ideals quickly gets a lot more complicated. We shall give criteria for similar equality to hold for pairs of filtration of ideals, and exhibit examples in which these equality are not necessarily true. The main result of this section is highlighted in \Cref{thm.asym_res_int} and \Cref{asym_res_int}.

The next lemma is a direct generalization of \cite[Lemma 4.1]{DFMS}. We include the proof for completeness.

\begin{lemma}[{See \cite[Lemma 4.1]{DFMS}}]\label{asp_res_bound}
	Let $\aa_\bullet$ and $\bb_\bullet$ be filtration of ideals in $S$. Suppose that $\{s_n\}_{n \in \NN}$ and $\{r_n\}_{n \in \NN}$ are sequences of positive integers such that $\lim\limits_{n \to \infty} s_n =\lim\limits_{n \to \infty} r_n =\infty,$ $\aa_{s_n} \subseteq \bb_{r_n}$ for all $n$, and $\lim\limits_{n\to \infty} \dfrac{s_n}{r_n} =h$ for some $h \in \mathbb{R}_{\geq 0}.$ Then, $\rhat(\aa_\bullet,\bb_\bullet) \leq h.$
\end{lemma}

\begin{proof}
	Suppose, by contradiction, that $\rhat(\aa_\bullet,\bb_\bullet)>h$. Then, there exist $s,r \in \mathbb{N}$ such that $h<\frac{s}{r}<\rhat(\aa_\bullet,\bb_\bullet)$ and  $\aa_{st}\not\subseteq \bb_{rt}$ for all $t\gg 1$. Let $t_{0} \in \mathbb{N}$ be such that   $\aa_{st}\not\subseteq \bb_{rt}$ for all $t\ge t_0$.
	
	Set $\epsilon = \frac{s}{r}-h$. Observe that there exists $n_0 \in \mathbb{N}$ such that  $\frac{s_n}{r_n}<h+\epsilon = \frac{s}{r}$ for all $n \geq n_0$. That is, $sr_n-rs_n>0$ for all $n \geq n_0$.  Since $\lim\limits_{n \rightarrow \infty} \frac{sr_n - rs_n}{rr_n} = \frac{s}{r} - h = \epsilon$ and $\lim\limits_{n \rightarrow \infty} rr_n = \infty$, it follows that $\lim\limits_{n \to \infty} (sr_n-rs_n) = \infty$. Particularly, this implies that for $n \gg 1$, we have $r_n \ge rt_0$ and $sr_n-rs_n > rs$.
	
	Consider the smallest $n$ such that $r_n\ge rt_0$ and $sr_n-rs_n>rs$, and the largest $t$ such that, with this value of $n$, we have $r_n\ge rt$. Note that, with these choices of $n$ and $t$, we have $t\ge t_0$ and  $r_{n}<r(t+1)$. Therefore, $sr_{n}<srt+sr$ which implies that $rs_n+rs<sr_n<srt+sr$. Thus, $s_{n}<st$. As a consequence, we get that $\aa_{st}\subseteq \aa_{s_n}\subseteq \bb_{r_n} \subseteq \bb_{rt}$, which is a contradiction as $ t \geq t_0$. Hence,  $\rhat(\aa_\bullet,\bb_\bullet) \le h.$
\end{proof}

The following general, yet simple, observation is the starting point to consider $\overline{\bb_\bullet}$ in place of $\bb_\bullet$.

\begin{theorem}\label{thm.asym_res_int}
	Let $\aa_\bullet, \bb_\bullet$ and $\bb_\bullet'$ be filtration of ideals in $S$.  Suppose that $\bb_\bullet \le \bb_\bullet'$ and $\R({\bb_\bullet'})$ is a finitely generated $\R(\bb_\bullet)$-module. Then, $$\rhat(\aa_\bullet,\bb_\bullet')=\rhat(\aa_\bullet,\bb_\bullet).$$
\end{theorem}

\begin{proof}
Since $ \bb_i \subseteq  \bb_i'$ for all $i$,  $\rhat(\aa_\bullet,{\bb_\bullet'}) \leq \rhat(\aa_\bullet,\bb_\bullet)$ by \Cref{lem.estimate}. If $\rhat(\aa_\bullet,\bb_\bullet') =\infty$, then we are done. Thus, we may assume that $\rhat(\aa_\bullet,\bb_\bullet') <\infty$.
	
Consider any $s,r \in \NN$ with $h=\dfrac{s}{r} > \rhat(\aa_\bullet,\bb_\bullet'). $ We claim that $\rhat(\aa_\bullet,\bb_\bullet) \leq h.$ This particularly shows that $\rhat(\aa_\bullet, \bb_\bullet) \le \rhat(\aa_\bullet, \bb_\bullet')$ and establishes the assertion.
	
Indeed, since $\R(\bb_\bullet')$ is a finitely generated $\R(\bb_\bullet)$-module, there exists a homogeneous set of generators $\{u_1,\ldots,u_m\}$ of $\R(\bb_\bullet')$ as an $\R(\bb_\bullet)$-module. Let $k=\max\{ \deg(u_i) ~\big|~ 1 \leq i \leq m\}$. Then, for $n \ge k$, we have
	$$\bb_n' = \sum_{i = 1}^m \bb_{n-\deg(u_i)} \bb_{\deg(u_i)}' \subseteq \bb_{n-k}.$$
	Since $\dfrac{s}{r} > \rhat(\aa_\bullet,\bb_\bullet')$, $\aa_{st} \subseteq \bb_{rt}'$ for infinitely many values of $t$. Let $\{t_n\}_{n \in \NN}$ be an increasing sequence of positive integers such that $t_1 \ge k$ and $\aa_{st_n} \subseteq \bb_{rt_n}'$ for all $n \in \NN$. It can be seen that $\aa_{st_n} \subseteq \bb_{rt_n}' \subseteq \bb_{rt_n-k}$ for all $n \in \NN$.  Now, let $s_n=st_n$ and $r_n=rt_n-k$ for all $n$. Then $\displaystyle \lim_{n \to \infty} s_n =\lim_{n \to \infty} r_n =\infty,$ $\aa_{s_n} \subseteq \bb_{r_n} $ for all $n$ and $\displaystyle \lim_{n \to \infty}\dfrac{s_n}{r_n}  =\dfrac{s}{r}.$ Thus, by \Cref{asp_res_bound}, $\rhat(\aa_\bullet,\bb_\bullet) \leq \dfrac{s}{r}=h$. The result is proved.
\end{proof}

With essentially the same proof, the hypothesis of \Cref{thm.asym_res_int} can be made slightly weaker as follows.

\begin{corollary} \label{cor.asym_res_int}
	Let $\aa_\bullet$ and $ \bb_\bullet$ be filtration of ideals in $S$. Let $\bb_\bullet'$ be a graded family such that $\bb_\bullet \le \bb_\bullet'$. Suppose that there exists $k \in \NN$, such that $\bb'_{i+k} \subseteq \bb_i$ for all $i \in \NN$. Then,
	$$\rhat(\aa_\bullet,\bb_\bullet')=\rhat(\aa_\bullet,\bb_\bullet).$$
\end{corollary}

The conclusion of \Cref{thm.asym_res_int} is not necessarily true without the hypothesis that $\aa_\bullet$, $\bb_\bullet$ and $\bb_\bullet'$ are filtration, as demonstrated in the following example.

\begin{example}
	\label{ex.notFiltration}
	Let $\kk$ be a field, $S=\kk[x,y]$, and set
	\begin{align*}
		\aa_1 = \bb_1 =(x^3,y^3), \bb_2=(x^4,x^3y,xy^3,y^4),\aa_2=(x,y)^4.
	\end{align*}
	With the convention that $\bb_i = (0)$ for $i < 0$ and $\bb_0 = S$, consider collections $\aa_\bullet, \bb_\bullet$ and $\bb_\bullet'$ of ideals in $S$ given as follows:
	\begin{align*}
		\bb_n &= \begin{cases}
			\bb_1, &\text{if $n\equiv 1 \quad \text{(modulo 3)}$}\\
			\bb_2, &\text{if $n\equiv 2 \quad \text{(modulo 3)}$}\\
			\bb_1\bb_2, &\text{if $n\equiv 0 \quad \text{(modulo 3)}$},
		\end{cases} \\
		\aa_n& =\begin{cases}
			\bb_1, &\text{if $n\equiv 1 \quad \text{(modulo 3)}$}\\
			\aa_2, &\text{if $n\equiv 2 \quad \text{(modulo 3)}$}\\
			\bb_1\aa_2, &\text{if $n\equiv 0 \quad \text{(modulo 3)}$},
		\end{cases}\\
		\bb'_n &= \bb_n+\bb_{n-2}\aa_2, \text{ for } n \ge 1.
	\end{align*}
	We shall show that $\bb_\bullet, \aa_\bullet,\bb'_\bullet$ are graded families, $\bb_\bullet \le \bb_\bullet'$, and $\R({\bb_\bullet'})$ is a finitely generated $\R(\bb_\bullet)$-module. Nevertheless, $\rhat(\aa,\bb')=-\infty  < \rhat(\aa,\bb)=\infty$, contradicting the conclusion of \Cref{thm.asym_res_int}.
	
We shall start with the observation that the following relations hold:
	\begin{enumerate}[\quad \rm (i)]
		\item $\bb_1^2\subseteq \bb_2\subseteq \bb_1$;
		\item $\bb_1= \aa_1, \bb_2\subseteq \aa_2$;
		\item $x^5y^2\in \bb_1\aa_2\setminus \bb_1\bb_2$;
		\item $\aa_2^2=\bb_2\aa_2=(x,y)^8$;
		\item $\aa_1^2\subseteq \aa_2, \aa_2^2\subseteq \aa_1=\bb_1$.
	\end{enumerate}
	
Indeed, (i), (ii) and (iv) follow by direction verification.
To see (iii), it suffices to note that $\bb_1\bb_2=(x^7,x^6y,x^4y^3,x^3y^4,xy^6,y^7)$ and $x^5y^2=x^3(x^2y^2)\in \bb_1\aa_2$.
In order to prove (v), we observe that, from (i),
	\[
	\aa_1^2 =\bb_1^2\subseteq \bb_2 \subseteq \aa_2.
	\]
Thus, it follows from (iv) that
	\[
	\aa_2^2=\aa_2\bb_2\subseteq \bb_2 \subseteq \bb_1=\aa_1,
	\]
which establishes (v).
	
We continue by claiming the following statements:
	\begin{enumerate}
		\item $\bb_\bullet, \aa_\bullet$ are graded families; in fact, $\bb_m\bb_n\subseteq \bb_{m+n}$ for all $m,n\in \ZZ$;
		\item  $\bb'_1=\bb_1, \bb_2 \subseteq \aa_2=\bb'_2$;
		\item $\bb'_n=\bb_n+\bb_{n-2}\aa_2=\bb_n+\bb_{n-2}\bb'_2$ for all $n\ge 1$;
		\item $\bb'_\bullet$ is a graded family and $\bb_\bullet \le \bb_\bullet'$;
		\item $\R(\bb_\bullet')$ is a finitely generated $\R(\bb_\bullet)$-module;
		\item $\aa_{3q}=\bb_1\aa_2=\bb'_{3n}, \bb_{3q}=\bb_1\bb_2$ for all $n,q\ge 1$;
		\item $\rhat(\aa,\bb')=-\infty$; and,
		\item $\rhat(\aa,\bb)=\infty$.
	\end{enumerate}

Indeed, (2) can be verified directly, and (3) follows from (2). We begin with (1).
Since $\bb_\bullet$ is periodic of period 3, to show that $\bb_\bullet$ is a graded family it suffices to check that
	\[
	\bb_1^2 \subseteq \bb_2, \bb_2^2\subseteq \bb_1, \bb_1\bb_2 \subseteq \bb_3.
	\]
This is a consequence of (i) and the fact that $\bb_3=\bb_1\bb_2$.
Similarly, to see that $\aa_\bullet$ is a graded family, it suffices to check that
	\[
	\aa_1^2 \subseteq \aa_2, \aa_2^2\subseteq \aa_1, \aa_1\aa_2 \subseteq \aa_3.
	\]
This follows from (v) and the fact that $\aa_1=\bb_1$ and $\aa_3=\bb_1\aa_2$.
The remaining assertion of (1) is clear using $\bb_i=(0)$ for $i<0$ and $\bb_0=S$.
	
For (4), it is clear from the definition that $\bb_\bullet \le \bb_\bullet'$. Thus, it remains to check that $\bb'_\bullet$ is a graded family. For $m,n\ge 1$, we have
	\begin{align*}
		\bb'_m\bb'_n &=(\bb_m+\bb_{m-2}\aa_2)(\bb_n+\bb_{n-2}\aa_2) \\
		&=\bb_m\bb_n+(\bb_m\bb_{n-2}+\bb_{m-2}\bb_n)\aa_2+\bb_{m-2}\bb_{n-2}\aa_2^2\\
		& \subseteq \bb_{m+n} +\bb_{m+n-2}\aa_2+\bb_{m-2}\bb_{n-2}\aa_2^2 \quad \text{(using (1))}\\
		&= \bb'_{m+n}+\bb_{m-2}\bb_{n-2}\bb_2\aa_2 \quad \text{(by definition of $\bb'_\bullet$ and (iv))}\\
		&\subseteq \bb'_{m+n}+\bb_{m+n-2}\aa_2 \quad \text{(using (1))}\\
		&=\bb'_{m+n} \quad \text{(using the definition of $\bb'_\bullet$)}.
	\end{align*}
	This is the desired containment.
	
To see (5), observe that by (3) we have $\bb'_n=\bb_n+\bb_{n-2}\bb'_2$ for all $n\ge 1$. Therefore, $\R(\bb')=\R(\bb)+\R(\bb)\bb'_2t^2$, which is a finitely generated $\R(\bb)$-module.
	
To prove (6), we remark that $\aa_{3q}=\bb_1\aa_2$ and $\bb_{3n}=\bb_1\bb_2$ by definition. Also,
	\begin{align*}
		\bb'_{3n}&=\bb_{3n}+\bb_{3n-2}\aa_2\\
		&=\bb_1\bb_2+\bb_1\aa_2 \quad \text{(by definition of $\bb_\bullet$)}.\\
		&= \bb_1\aa_2 \quad \text{(as $\bb_2\subseteq \aa_2$)}.
	\end{align*}
	
To establish (7), we claim that there is no pair of positive integers $(s,r)$ such that $\aa_{sn}\not\subseteq \bb'_{rn}$ for all $n\gg 0$. Indeed, for any $n\ge 1$, by (6), we have
	\[
	\aa_{3sn}=\bb_1\aa_2= \bb'_{3rn}.
	\]
Hence, $\rhat(\aa_\bullet,\bb'_\bullet)=-\infty$.  On the other hand, for all $q\ge 1$, by (6) and (iii), we get
	\[
	x^5y^2\in \aa_{3qn} \setminus \bb_{3n} \quad \text{for all $n\ge 1$}.
	\]
Therefore, $\rhat(\aa_\bullet,\bb_\bullet) \ge (3q)/3=q$, for all $q\ge 1$. This yields (8) and the desired computation.
\end{example}

As a corollary to \Cref{thm.asym_res_int}, we extend \cite[Proposition 4.2]{DFMS} to give a criterion when the asymptotic resurgence number of $(\aa_\bullet, \bb_\bullet)$ remains unchanged if we replace the family $\bb_\bullet$ by $\overline{\bb_\bullet}$. Before presenting this result, in \Cref{asym_res_int}, we remark that $\R(\overline{\bb_\bullet})$ is not necessarily a finitely generated module over $\R(\bb_\bullet)$, even when $\R(\bb_\bullet)$ is Noetherian, as discussed below.

For a Noetherian ring $B$ and an ideal $\nn \subseteq B$, let $\widehat{B}^\nn$ be the $\nn$-adic completion of $B$. We say that a Noetherian local ring $(S,\mm)$ is \emph{analytically unramified} if $\widehat{S}^{\mm}$ is a reduced ring, and \emph{analytically ramified} otherwise. The following example is a well-known construction (due to Nagata) of a one-dimensional Noetherian local domain that is analytically ramified. Statement (4) of \Cref{ex_Nagata} also shows that $\R(\overline{\bb^\bullet})$  need not be a finitely generated $\R(\bb^\bullet)$-module in general, even when $S$ is a one-dimensional Noetherian local domain, and $\bb$ is a principal ideal.

\begin{example}[Nagata]
\label{ex_Nagata}
Let $p$ be a prime number. Denote $\kk=\mathbb{F}_p(t_1,t_2,\ldots)$, and denote by $\kk^p$ the subfield of $p$-th powers in $\kk$. Let $A=\kk^p[[x]][\kk] \subseteq \kk[[x]]$, i.e. $A$ is the subring of $\kk[[x]]$ generated by elements of $\kk^p[[x]] \cup \kk$. Since a power series $\sum_{i=0}^\infty a_ix^i\in \kk[[x]]$ belongs to $A$ if and only if it is a $\kk$-linear combination of finitely many elements in $\kk^p[[x]]$, we deduce that
\begin{equation}
\label{eq_A_description}
A= \left\{\sum_{i=0}^\infty a_ix^i\in \kk[[x]] ~\Big|~ [\kk^p(a_0,a_1,a_2\ldots):\kk^p] <\infty \right \}.
\end{equation}

Let $f=\sum\limits_{i=1}^\infty t_ix^i \in \kk[[x]]\setminus A$. Set $S=A[f] \subseteq \kk[[x]]$ and let $\bb=xS\subseteq S$. We claim that the following statements hold.
\begin{enumerate}
\item $(A, xA)$ is a discrete valuation ring and $\widehat{A}^{xA}= \kk[[x]]$.
 \item  $S$ is a one-dimensional Noetherian domain, and it is a finitely generated $A$-module.
 \item $S$ is a local ring with the unique maximal ideal $\mm=(x,f)S$.
 \item For all $n\ge 1$, we have $f-\sum\limits_{i=1}^{n-1} t_ix^i \in \overline{\bb^n} \setminus \bb$. In particular, there does not exist an integer $c\ge 1$ such that $\overline{\bb^n} \subseteq \bb^{n-c}$ for all $n\ge c$, and $\R(\overline{\bb^\bullet})$  is not a finitely generated $\R(\bb^\bullet)$-module.
 \item There is an isomorphism $S\cong A[T]/(T^p-f^p)$.
 \item There is an isomorphism $\widehat{S}^\mm \cong \dfrac{\kk[[x]][T]}{((T-f)^p)}$. In particular, $S$ is \emph{not} analytically unramified.
\end{enumerate}
We shall give a down-to-earth treatment of these facts, since this example is of importance in the present manuscript. For all the statements, except (4), the experienced reader may follow the sketch given by Nagata in \cite[pp.\,205--206]{Na62}.

\textit{Statement (1).}  Applying \Cref{lem_intergralext_localitytransfer}(1) for the integral extension $A\subseteq B=\kk[[x]]$, we deduce that $(A, xB\cap A)$ is a local domain. Using \eqref{eq_A_description}, we deduce that $xB\cap A=xA$. Therefore, $A$ is a discrete valuation ring with the unique maximal ideal $xA$.  Completing the injection $A\subseteq B$ at $xA$, we get a map $\widehat{A}^{xA} \to \widehat{B}^{xA}=\widehat{B}^{xB}=\kk[[x]]$. This map is clearly surjective as $\kk \cup \{x\} \subseteq A$. Since $\widehat{A}^{xA}$  and $\kk[[x]]$ are both one dimensional regular local rings,  the map is an isomorphism, giving $\widehat{A}^{xA}= \kk[[x]]$.

\textit{Statement (2).} The ring $S=A[f]$ is Noetherian by Hilbert's Basis Theorem. Moreover, $A\to S=A[f]$ is an integral extension as $f^p\in x^pA$, so $S$ has dimension 1 and it is a finitely generated $A$-module.

\textit{Statement (3).} This follows by applying \Cref{lem_intergralext_localitytransfer}(2) for the map $A\subseteq S=A[f]$, noting that $f^p\in x^pA$.

\textit{Statement (4).} We have
\[
\left(f-\sum\limits_{i=1}^{n-1} t_ix^i\right)^p=\sum\limits_{j=n}^\infty t_j^px^{pj} \in x^{pn}S=\bb^{pn},
\]
so $f-\sum\limits_{i=1}^{n-1} t_ix^i \in \overline{\bb^n}$. We shall show that $f-\sum\limits_{i=1}^{n-1} t_ix^i\notin xS$. Indeed, assume the contrary, then $f\in xS$. Consider the finitely generated $A$-module $xA+fS \subseteq S$. Since $f\in xS$ and $S=A+fS$,
\[
xA+fS \subseteq xA+xS=xA+x(A+fS)=xA+x(xA+fS).
\]
By Nakayama's lemma over the ring $(A,xA)$, we deduce that $xA+fS=xA$. Hence, $f\in xA$. This is a contradiction, exhibiting that $f-\sum\limits_{i=1}^{n-1} t_ix^i\notin xS$. The remaining assertions are immediate.

\textit{Statement (5).} Let $J$ be the kernel of the natural surjection $A[T]\to S, T\mapsto f$. By (2) and the fact that $\dim A[T]=2$, we get $J$ is a prime ideal of height 1. Since $A$ is an UFD, so is $A[T]$. Thus, $J$ is a principal ideal. Note that $f^p$ is not a $p$-th power in $A$, and so $T^p-f^p\in A[T]$ is irreducible. This forces the containment $(T^p-f^p) \subseteq J$ to be an equality, in other words, $S\cong A[T]/(T^p-f^p)$.

\textit{Statement (6).} Note that $f^p\in xS$, so $xS \subseteq \mm=(x,f)$ and $\mm^p\subseteq xS$. Therefore, the adic topologies defined by $\mm$ and $xS$ are the same. Hence $\widehat{S}^\mm \cong \widehat{S}^{xS}$. We have an exact sequence of finitely generated $A[T]$-modules:
\[
0\to A[T] \xrightarrow{\cdot (T^p-f^p)} A[T] \to S \to 0.
\]
Completing  with respect to $xA[T]$ and noting that $\widehat{A[T]}^{xA[T]}\cong (\widehat{A}^{xA})[T]=\kk[[x]][T]$ we get an exact sequence
\[
0\to  \kk[[x]][T] \xrightarrow{\cdot (T^p-f^p)} \kk[[x]][T] \to \widehat{S}^\mm \to 0.
\]
Hence, as $\chara \kk=p$,
\[
 \widehat{S}^\mm \cong \kk[[x]][T]/(T^p-f^p)= \kk[[x]][T]/((T-f)^p).
\]
\end{example}

To complete the arguments in \Cref{ex_Nagata}, we need to establish the following lemmas on the transfer of locality along integral ring extensions, that are perhaps folklore results.

\begin{lemma}
\label{lem_intergralext_localitytransfer}
Let $A\subseteq B$ be an integral ring extension.
\begin{enumerate}[\quad \rm (1)]
 \item Assume that $B$ is a local ring with the maximal ideal $\nn$. Then, $A$ is also a local ring with the unique maximal ideal $\nn \cap A$.
 \item Assume that $A$ is a local ring with the maximal ideal $\mm$, and $B$ is generated over $A$ by finitely many elements $f_1,\ldots,f_n$ that belong to $\sqrt{\mm B}$. Then, $B$ is also a local ring with the unique maximal ideal $\mm B+(f_1,\ldots,f_n)$.
\end{enumerate}
\end{lemma}

\begin{proof}
(1) Since $A\subseteq B$ is integral and $\nn$ is maximal ideal of $B$, $\nn \cap A$ is a maximal ideal of $A$. By lying over, for any maximal ideal $\mm$ of $A$, there exists a prime ideal $\qq$ of $B$ such that $\mm=\qq \cap A$. Since the extension $A/\mm = A/(\qq\cap A) \to B/\qq$ remains integral, $A/\mm$ is a field and $B/\qq$ is a domain, we deduce that $B/\qq$ is a field. Hence $\qq=\nn$. Therefore $(A,\nn \cap A)$ is a local ring.

(2) For any maximal ideal $\nn$ of $B$, the integral extension $A\subseteq B$ implies that $\nn\cap A$ is a maximal ideal of $A$. Hence $\nn \cap A=\mm$, $\mm B\subseteq \nn$, and in particular,
\[
\mm B\cap A=\mm.
\]
This yields an integral ring extension $A/\mm \subseteq B/\mm B$. Since $\mm B$ is contained in the Jacobson radical of $B$, replacing $A,B$ by $A/\mm, B/\mm B$ respectively, we may assume that $A=\kk$ is a field, $\mm=0$. Now $B=\kk[f_1,\ldots,f_n]$ and $f_i$ is nilpotent for each $i$, so $B$ is an artinian local ring with the unique maximal ideal $(f_1,\ldots,f_n)$, as desired.
\end{proof}

\begin{remark}
The ring $S$ in Nagata's \Cref{ex_Nagata} has characteristic $p>0$, but the restriction to positive characteristic turns out to be inessential. Using K\"ahler differentials, Ferrand and Raynaud \cite[Proposition 3.1]{FR70} (see, e.g., \cite[Section 109.16]{StacksEx}) have constructed a Noetherian local domain of dimension one containing $\mathbb{Q}$ that shares strikingly similar properties with Nagata's example (e.g. that of being analytically ramified). In terms of exposition, a slight advantage of Nagata's example over Ferrand--Raynaud's one is that it is more explicit and computationally simpler.
\end{remark}

We have seen from Example \ref{ex_Nagata} that $\R(\overline{\bb_\bullet})$ is not necessarily a finitely generated $\R(\bb_\bullet)$-module, even when $\R(\bb_\bullet)$ is Noetherian. On the other hand, if the ring $S$ is nice enough and the Rees algebra $\R(\bb_\bullet)$ is Noetherian, then $\R(\overline{\bb_\bullet})$ is a finitely generated $\R(\bb_\bullet)$-module, as discussed in the following remark.

\begin{remark} Let $S$ be a regular ring or, more generally, an analytically unramified semi-local ring. Let $\bb_\bullet$ be a filtration of ideals in $S$. Then, the integral closure of $\R(\bb_\bullet)$ in $S[t]$ is a $\mathbb{N}$-graded ring, i.e., there exists a graded family of ideals $\cc_\bullet$ in $S$ such that
$$\displaystyle \overline{\R(\bb_\bullet)}=\bigoplus_{i\ge 0} \cc_i t^i.$$
It should be noted that $\overline{\bb_i} \subseteq \cc_i$ for all $i \ge 1$.

It follows from \cite{FE, Rat} that, if $\bb_\bullet$ is a Noetherian filtration then $\overline{\R(\bb_\bullet)}$ is a finitely generated $\R(\bb_\bullet)$-module, and therefore, $\R(\overline{\bb_\bullet})$ is a finitely generated $\R(\bb_\bullet)$-module.
\end{remark}

Under the hypothesis that $\R(\overline{\bb_\bullet})$ is a finitely generated $\R(\bb_\bullet)$-module, the following corollary of \Cref{thm.ReesVal} generalizes \cite[Proposition 4.2]{DFMS}.

\begin{corollary}
	\label{asym_res_int}
	Let $\aa_\bullet$ and $\bb_\bullet$ be filtration of ideals in $S$.  Suppose that $\R(\overline{\bb_\bullet})$ is a finitely generated $\R(\bb_\bullet)$-module. Then, $$\rhat(\aa_\bullet,\overline{\bb_\bullet})=\rhat(\aa_\bullet,\bb_\bullet).$$
\end{corollary}

\begin{proof}
	The assertion is a direct consequence of \Cref{thm.asym_res_int}, by letting $\bb_\bullet' = \overline{\bb_\bullet}$.
\end{proof}

Recall that if $S$ is of prime characteristic, then for an ideal $I \subseteq S$, $I^*$ denotes its tight closure. We obtain the following result for filtration of tight closures.

\begin{corollary}
Suppose that $S$ is a Noetherian ring of prime characteristic $p> 0$. Let $\aa_\bullet, \bb_\bullet$ be filtration of ideals, and let $\bb$ be a nonzero proper ideal in $S$.
	\begin{enumerate}
	\item We have that $\rhat(\aa_\bullet,\overline{\bb^\bullet})=\rhat(\aa_\bullet,\left(\bb^\bullet\right)^{*}),$ where $\left(\bb^\bullet\right)^{*} =\{ \left(\bb^n\right)^{*}\}_{n\ge 0}$.
	\item If $\R(\overline{\bb_\bullet})$ is a finitely generated $\R(\bb_\bullet)$-module, then $\rhat(\aa_\bullet,\overline{\bb_\bullet})=\rhat(\aa_\bullet,\bb_\bullet^{*})=\rhat(\aa_\bullet,\bb_\bullet),$ where $\bb_\bullet^{*} =\{ \bb_n^{*}\}_{n\ge 0}$.
	\end{enumerate}
\end{corollary}

\begin{proof}
It is well know that for any ideal $I\subseteq S$, $I \subseteq I^* \subseteq \overline{I}$. Therefore,  $\left(\bb^\bullet\right)^{*} \le \overline{\bb^\bullet}$ and $\bb_\bullet \le \bb_\bullet^{*} \le \overline{\bb_\bullet}.$ Now, the second assertion is a direct consequence of \Cref{thm.asym_res_int}. Next, by \cite[Theorem 13.2.1]{SH2006}, there exists a positive integer $k$ such that $\overline{\bb_{n+k}} \le \left(\bb^n\right)^*$ for all $n$. The first assertion now follows from \Cref{cor.asym_res_int}.
\end{proof}

As a consequence of \Cref{asym_res_int}, we obtain a generalization of \cite[Theorem 3.7]{DD2020} on the rationality of the usual asymptotic resurgence number. The rationality of the usual resurgence number in \cite[Theorem 3.7]{DD2020} will be generalized in the last section.

\begin{corollary}
	\label{cor.rationalAR}
	Let $S$ be a Noetherian domain, and let $\aa_\bullet$ and $\bb_\bullet$ be filtration of nonzero ideals in $S$. Suppose that $\R^{[k]}(\aa_\bullet)$ and  $\R^{[\ell]}(\bb_\bullet)$ are standard graded $S$-algebras for some $k$ and $\ell$, and that $\R(\overline{\bb_\bullet})$ is a finitely generated $\R(\bb_\bullet)$-module. Then, $\rhat(\aa_\bullet, \bb_\bullet) = \rhat(\aa_\bullet, \overline{\bb_\bullet})$ is either infinity or a rational number.
\end{corollary}

\begin{proof}
	The assertion follows from \Cref{thm.ReesVal} and \Cref{asym_res_int}.
\end{proof}

\begin{question} Characterize when the equality $\rhat(\aa_\bullet, \bb_\bullet) = \rhat(\aa_\bullet, \overline{\bb_\bullet})$ holds.
\end{question}

%%%%%%%%%%%%%%%%%%%%%%%%%%%
%%%%%%%%%%%%%%%%%%%%%%%%%%%

\section{$\bb$-equivalent families} \label{sec.equiv}

In this section, we focus on the situation where the graded family $\bb_\bullet$ is $\bb$-equivalent for some ideal $\bb$ in $S$. Familiar examples include the cases when $\bb_\bullet$ is the family of ordinary powers or their integral closures of a given ideal. We shall show that under this condition and a mild assumption on the ring $S$, the following natural generalization of known equality between the resurgence and asymptotic resurgence numbers of an ideal hold:
$$\rhat(\aa_\bullet, \bb_\bullet) = \rhat(\aa_\bullet, \overline{\bb_\bullet}) = \rho(\aa_\bullet, \overline{\bb_\bullet}).$$
The main result of this section is \Cref{thm.b1equivalent}; see also \Cref{thm.rhoInt}.

We begin with a technical lemma, which is similar to \cite[Lemma 4.12]{DFMS}.

\begin{lemma}[{Compare with \cite[Lemma 4.12]{DFMS}}]
	\label{tech_lemm} Let $S$ be a Noetherian domain, and let $\aa_\bullet$ and $\bb_\bullet$ be graded families of nonzero ideals  in $S$.
	\begin{enumerate}
		\item Assume that $\vhat(\bb_\bullet) =v(\bb_1)$ for all $v \in \RV(\bb_1)$. If $\aa_s \not\subseteq \overline{\bb_r}$ then $\dfrac{s}{r} < \rhat(\aa_\bullet,\overline{\bb_\bullet})$. In particular, $\rho(\aa_\bullet,\overline{\bb_\bullet})=\rhat(\aa_\bullet,\overline{\bb_\bullet}).$
		\item Assume that $\R^{[k]}(\bb_\bullet)$ is a standard graded $S$-algebra for some $k$. If $\dfrac{s}{r} < \rhat(\aa_\bullet,\overline{\bb_\bullet})$, then $\aa_{st} \not\subseteq \overline{\bb_{rt}}$ for $t \gg 1.$
	\end{enumerate}
\end{lemma}

\begin{proof}
	(1)	If $\aa_s \not\subseteq \overline{\bb_r}$ then, as $\bb_1^r \subseteq \bb_r$, we have $\aa_s \not\subseteq \overline{\bb_1^r}$. It follows from  \cite[Theorem 6.8.3 and Remark 10.1.4]{SH2006} that  there exists a valuation $v \in \RV(\bb_1)$ such that $v(\aa_s)<v(\bb_1^r).$ This, together with \Cref{low_bou}, implies that
	$$\dfrac{s}{r} < \dfrac{s}{r}  \dfrac{v(\bb_1^r)}{v(\aa_s)} = \dfrac{s}{r} \dfrac{r v(\bb_1)}{v(\aa_s)} = \dfrac{v(\bb_1)}{v(\aa_s)/s} \leq \dfrac{\vhat(\bb_\bullet)}{\vhat(\aa_\bullet)} \le \rhat(\aa_\bullet,\overline{\bb_\bullet}).$$
	The proof of (1) is completed.
	
	(2)	Suppose that $\dfrac{s}{r} < \rhat(\aa_\bullet,\overline{\bb_\bullet}) .$ Then, by Theorem \ref{thm.ReesVal}, there exists $v \in \RV(\bb_k)$ such that $\dfrac{s}{r} < \dfrac{\vhat(\bb_\bullet)}{\vhat(\aa_\bullet)}.$ Therefore, $\dfrac{s}{r} < \dfrac{\vhat(\bb_\bullet)}{v(\aa_{st})/st}$ for $t \gg 1$. Thus, for $t\gg 1$, we have $v(\aa_{st}) < rt  \vhat(\bb_\bullet) \leq v(\bb_{rt}).$ Now, it follows from \cite[Theorem 6.8.3]{SH2006} that $\aa_{st} \not \subseteq \overline{\bb_{rt}}$ for $t \gg 1$. This establishes (2).
\end{proof}

The hypothesis of \Cref{tech_lemm}.(1) holds for a large class of graded families, those that are $\bb$-equivalent, as defined below.

\begin{definition}
Let $\bb \subseteq S$ be an ideal. 	We call a graded family $\bb_\bullet = \{\bb_i\}_{i \ge 1}$ of ideals in $S$ a \emph{$\bb$-equivalent} family if there exists a positive integer $k$ such that, for all $i \ge 1$,
	$$\bb_{i+k} \subseteq \bb^{i}\subseteq \bb_i.$$
\end{definition}
\noindent Note that if $\bb_\bullet$ is a $\bb$-equivalent family then we necessarily have $\bb_{1+k} \subseteq \bb \subseteq \bb_1$ for some $k \in \NN$.

As a consequence of \Cref{tech_lemm}, we obtain the following theorems, which generalize \cite[Proposition 2.1]{DD2020} and address the computation and rationality of resurgence and asymptotic resurgence numbers.

\begin{theorem}
	\label{thm.rhoReesVal}
	Let $S$ be a Noetherian domain and let $\aa_\bullet$ and $\bb_\bullet$ be graded families of nonzero ideals in $S$ such that the following conditions hold:
	\begin{enumerate}
		\item $\rhat(\aa_\bullet, \overline{\bb_\bullet}) < \rho(\aa_\bullet,\bb_\bullet);$
		\item $\vhat(\bb_\bullet) =v(\bb_1)$ for all $v \in \RV(\bb_1)$;
		\item there exists a positive integer $k$ so that $\overline{\bb_{i+k}} \subseteq \bb_i$ for all $i \in \NN$.
	\end{enumerate}
	Then, there exist positive integers $s_0,r_0$ such that $\aa_{s_0} \not\subseteq \bb_{r_0}$,  $\dfrac{s_0}{r_0} > \rhat(\aa_\bullet, \overline{\bb_\bullet})$ and
	$$\rho(\aa_\bullet, {\bb_\bullet}) =\max \left \{ \dfrac{s}{r} ~\Big|~ 1 \leq r <N, 1 \leq s < (r+k)\rhat(\aa_\bullet, \overline{\bb_\bullet}) \text{ and } \aa_{s} \not\subseteq \bb_r \right \},$$
	where
	$$N = \dfrac{k  \rhat(\aa_\bullet, \overline{\bb_\bullet})}{\dfrac{s_0}{r_0}-\rhat(\aa_\bullet, \overline{\bb_\bullet})}.$$
	Moreover, under these conditions, $\rho(\aa_\bullet,\bb_\bullet)$ is a positive  rational number.
\end{theorem}

\begin{proof}
	Since $\rhat(\aa_\bullet, \overline{\bb_\bullet}) < \rho(\aa_\bullet,\bb_\bullet)$,  there exist positive integers $s_0,r_0$ such that $\aa_{s_0} \not\subseteq \bb_{r_0}$ and  $\dfrac{s_0}{r_0} > \rhat(\aa_\bullet, \overline{\bb_\bullet})$.
	
	Consider any $s,r \in \NN$ such that $\aa_s \not\subseteq \bb_r.$ Since $\overline{\bb_{r+k}} \subseteq \bb_r$, we get that $\aa_s \not\subseteq \overline{\bb_{r+k}}.$ By Lemma \ref{tech_lemm}, we have $\dfrac{s}{r+k} < \rhat(\aa_\bullet, \overline{\bb_\bullet})$ which implies that $\dfrac{s}{r} < \left ( 1+ \dfrac{k}{r}\right)\rhat(\aa_\bullet, \overline{\bb_\bullet}).$ If $r \geq N$, then we get that
	$$\dfrac{s}{r} < \left ( 1+ \dfrac{k}{r}\right)\rhat(\aa_\bullet, \overline{\bb_\bullet}) \le \dfrac{s_0}{r_0} \le \rho(\aa_\bullet, \bb_\bullet).$$
	Thus, to obtain $ \rho(\aa_\bullet,\bb_\bullet) $, it is enough to consider $r<N.$ Since $\dfrac{s}{r+k} < \rhat(\aa_\bullet, \overline{\bb_\bullet})$, we have $s <(r+k) \rhat(\aa_\bullet, \overline{\bb_\bullet})$. Hence, the assertion follows.
\end{proof}

\begin{theorem}
	\label{thm.rhatReesVal}
	Let $S$ be a Noetherian domain and let $\aa_\bullet$ and $\bb_\bullet$ be graded families of nonzero ideals in $S$ such that the following conditions hold:
	\begin{enumerate}
		\item $\rhat(\aa_\bullet, \overline{\bb_\bullet}) < \rhat(\aa_\bullet,\bb_\bullet);$
		\item $\vhat(\bb_\bullet) =v(\bb_1)$ for all $v \in \RV(\bb_1)$;
		\item there exists a positive integer $k$ so that $\overline{\bb_{i+k}} \subseteq \bb_i$ for all $i \in \NN$.
	\end{enumerate}
	Then, there exist positive integers $s_0,r_0$ such that $\aa_{s_0t} \not\subseteq \bb_{r_0t}$ for $t\gg 1$,  $\dfrac{s_0}{r_0} > \rhat(\aa_\bullet, \overline{\bb_\bullet})$ and
	$$\rhat(\aa_\bullet, {\bb_\bullet}) =\max \left \{ \dfrac{s}{r} ~\Big|~ 1 \leq r <N, 1 \leq s < (r+k)\rhat(\aa_\bullet, \overline{\bb_\bullet}) \text{ and } \aa_{st} \not\subseteq \bb_{rt} \text{ for } t\gg 1 \right \},$$
	where
	$$N = \dfrac{k  \rhat(\aa_\bullet, \overline{\bb_\bullet})}{\dfrac{s_0}{r_0}-\rhat(\aa_\bullet, \overline{\bb_\bullet})}.$$
	Moreover, under these conditions, $\rhat(\aa_\bullet,\bb_\bullet)$ is a positive rational number.
\end{theorem}

\begin{proof} The proof goes along the same line of arguments as that of \Cref{thm.rhoReesVal}. We shall leave the details to the interested readers.
\end{proof}

Theorem \ref{thm.rhatReesVal} as well as Corollary \ref{cor.rationalAR} provide sufficient conditions for which the asymptotic resurgence number $\rhat(\aa_\bullet, {\bb_\bullet})$ is rational while Theorem $\ref{thm.rhoReesVal}$ provide sufficient conditions for resurgence number $\rho(\aa_\bullet, {\bb_\bullet})$. In the last section of the paper, we will continue to investigate this problem.

\begin{remark} \label{rmk.regularRing}
	If $S$ is a analytically unramified local ring, then the hypothesis (2) and (3) of Theorems \ref{thm.rhoReesVal} and \ref{thm.rhatReesVal} holds for  any $\bb$-equivalent filtration of ideals in $S$, where $\bb$ is a nonzero proper ideal in $S$. This is because of \Cref{lem.a1} below and \cite[Theorem 9.2.1]{SH2006}.
\end{remark}

\begin{lemma}\label{lem.a1}
Let $S$ be a Noetherian domain, let $\bb$ be an ideal, let $\bb_\bullet$ be a graded family of ideals in $S$, and let $v$ a valuation of $K$ supported on $S$. Suppose that $\bb_\bullet$ is $\bb$-equivalent. Then $\vhat(\bb_\bullet) =v(\bb)$.
\end{lemma}

\begin{proof}
	Since $\bb_{i +k}\subseteq \bb^i \subseteq \bb_i$, we have $v(\bb_i) \le i v(\bb)=v(\bb^i) \leq v(\bb_{i+k})$  for all $i \in \NN$. Therefore,
	$$ \dfrac{i}{i+k}v(\bb) \leq  \dfrac{v(\bb_{i+k})}{i+k}.$$
	This implies that $v(\bb) \leq \vhat(\bb_\bullet) .$ The assertion follows as $\vhat(\bb_\bullet) \le v(\bb_1) \le v(\bb)$.
\end{proof}

As a consequence of \Cref{asym_res_int}, \Cref{tech_lemm} and \Cref{lem.a1}, we obtain a generalization of \cite[Proposition 4.2 and Corollary 4.14]{DFMS}.

\begin{corollary}
	\label{asym_res_ideal} Let $S$ be a domain that belongs to one of the following types:
\begin{enumerate}
\item complete local Noetherian ring;
\item finitely generated over a field or over $\ZZ$;
\item or, more generally, finitely generated over a Noetherian integrally closed domain $R$ satisfying the property that every finitely generated $R$-algebra has a module-finite integral closure.
\end{enumerate}
	Let $\aa_\bullet$ be a filtration of nonzero ideals, and let $\bb$ be a nonzero ideal in $S$. Then,
	$$\rhat(\aa_\bullet,\bb^\bullet) = \rhat(\aa_\bullet,\overline{\bb^\bullet}) = \rho(\aa_\bullet,\overline{\bb^\bullet}).$$
\end{corollary}

\begin{proof} By \cite[Proposition 5.3.4]{SH2006}, we have that $\R(\overline{\bb^\bullet})$ is a finitely generated module over $\R(\bb^\bullet)$. Thus, it follows from \Cref{asym_res_int} that $\rhat(\aa_\bullet,\overline{\bb^\bullet})=\rhat(\aa_\bullet,\bb^\bullet).$ The second equality follows from \Cref{tech_lemm}(1) and \Cref{lem.a1}.
\end{proof}

We proceed to our next main results of this section, which give criteria for the equality in \Cref{asym_res_ideal} to hold when the filtration of ordinary powers $\bb^\bullet$ is replaced by a more general graded family $\bb_\bullet$ of ideals.

\begin{theorem}
	\label{thm.b1equivalent}
Let $S$ be a domain as in \Cref{asym_res_ideal}.
Let $\aa_\bullet$ be a filtration and let $\bb_\bullet$ be a graded family of nonzero ideals in $S$. Suppose that $\bb_\bullet$ is $\bb$-equivalent for some ideal $\bb \subseteq S$. Then,
$$\rhat(\aa_\bullet,\bb_\bullet) = \rhat(\aa_\bullet,\overline{\bb_\bullet}) = \rho(\aa_\bullet,\overline{\bb_\bullet})= \rhat(\aa_\bullet,\bb^\bullet).$$
\end{theorem}

 \begin{proof}

 We first claim that $\rhat(\aa_\bullet,{\bb_\bullet}) = \rhat(\aa_\bullet,\bb^\bullet)$ and $\rhat(\aa_\bullet,\overline{\bb_\bullet}) = \rhat(\aa_\bullet,\overline{\bb^\bullet})$. Since $\bb_\bullet$ is $\bb$-equivalent, there exists a positive  integer $k$ such that $\bb_{i+k} \subseteq \bb^i$ for all $i \in \mathbb{N}$ and $\bb^{\bullet} \le \bb_\bullet.$ Therefore, by \Cref{cor.asym_res_int}, $\rhat(\aa_\bullet,{\bb_\bullet}) = \rhat(\aa_\bullet,\bb^\bullet)$.  Next, by \cite[Proposition 5.3.4]{SH2006},  there exists a positive integer $l$ such that $\overline{\bb^{i+l}} \subseteq \bb^i$ for all $i \in \NN$. Thus, $\overline{\bb_{i+l+k}} \subseteq \overline{\bb^{i+l}} \subseteq \bb^i$ for all $i \in \mathbb{N}$ and $\bb^{\bullet} \le \overline{\bb_\bullet}.$ Again, by \Cref{cor.asym_res_int}, $\rhat(\aa_\bullet,\overline{\bb_\bullet}) = \rhat(\aa_\bullet,{\bb^\bullet})$.

 Now, applying \Cref{asym_res_ideal}, we get $$\rhat(\aa_\bullet, \bb_\bullet) =\rhat(\aa_\bullet, \bb^\bullet)= \rho(\aa_\bullet, \overline{\bb^\bullet})=\rhat(\aa_\bullet,\overline{\bb^\bullet}) = \rhat(\aa_\bullet, \overline{\bb_\bullet}).$$
Moreover, since $\overline{\bb^n} \subseteq  \overline{\bb_n}$ for $n \ge 1$, it follows from \Cref{lem.estimate} again that $\rho(\aa_\bullet,\overline{\bb_\bullet}) \le \rho(\aa_\bullet,\overline{\bb^\bullet})$. Particularly,
$$\rhat(\aa_\bullet,\overline{\bb_\bullet}) \le \rho(\aa_\bullet,\overline{\bb_\bullet}) \le \rho(\aa_\bullet,\overline{\bb^\bullet})=\rhat(\aa_\bullet,\overline{\bb_\bullet}).$$
This proves the second desired equality.
 \end{proof}

In general, we do not expect the equalities in \Cref{asym_res_int}, \Cref{asym_res_ideal}, and \Cref{thm.b1equivalent} to hold even  when $S$ is a Noetherian local domain of dimension 1 and $\bb_\bullet$ is the filtration of ordinary powers of a principal ideal, as exhibited by the following example.

\begin{example}\label{ex.rhatNOTbar}
Use the notations of \Cref{ex_Nagata}. Let $\aa_\bullet,\bb_\bullet$ be  filtration given by $\aa_n=\overline{\bb^n}$ and $\bb_n=\bb^n$. We claim that $\rhat(\aa_\bullet,\bb_\bullet)=\infty > \rhat(\aa_\bullet,\overline{\bb_\bullet})=1$.

Indeed, since $\bb_\bullet$ is a filtration and by part (4) in \Cref{ex_Nagata},  $\aa_{nt}=\overline{\bb^{nt}}\not\subseteq \bb^t$ for all $n, t\ge 1$,  we get $\rhat(\aa_\bullet,\bb_\bullet)=\infty$. Note that $\aa_\bullet=\overline{\bb_\bullet}$. We claim that $\aa_n\subseteq \aa_m$ if and only if $n\ge m$. The ``if'' part is clear. Assume that $n<m$, we show that
$$
f-\sum\limits_{i=1}^{n-1} t_ix^i \in \aa_n \setminus \aa_m.
$$
By part (4) in \Cref{ex_Nagata}, it remains to show that $f-\sum\limits_{i=1}^{n-1} t_ix^i \notin \overline{\bb^m}=\overline{(x^m)S}$. Assume the contrary. Since $S\subseteq B=k[[x]], \overline{(x^m)S} \subseteq \overline{(x^m)B}=(x^m)B$. Hence
$f-\sum\limits_{i=1}^{n-1} t_ix^i \in (x^m)B$, a contradiction. Hence $\aa_n\subseteq \aa_m$ if and only if $n\ge m$. It then follows that given $s,r\ge 1$, we have $\aa_{sn}\not\subseteq \aa_{rn}$ for all $n\gg 1$ if and only if $s<r$. Thus $\rhat(\aa_\bullet,\overline{\bb_\bullet})=\rhat(\aa_\bullet,\aa_\bullet)=1$.
\end{example}

As remarked in \Cref{ex.rhatNOTbar}, we do not expect the equality $\rhat(\aa_\bullet,\bb_\bullet) = \rhat(\aa_\bullet,\overline{\bb_\bullet})$ in general. Furthermore, even if $\rhat(\aa_\bullet,\bb_\bullet) = \rhat(\aa_\bullet,\overline{\bb_\bullet})$ holds, it is not necessarily the case that this value is equal to $\rho(\aa_\bullet,\overline{\bb_\bullet})$, as seen in the following example.

 \begin{example}
 \label{ex.rhoNOTequalBAR}
 Let $S=\kk[x,y]$, and let $\aa_n=(x,y)^n$ for $n \ge 1$. Clearly, $\aa_\bullet $ is a filtration of ideals in $S$ and $\ahat(\aa_\bullet) = 1 = \ahat(\overline{\aa_\bullet})$.

 Let $\cc=(x)$. Let $\bb_1=\cc+(y^2)$ and more generally for each $n \geq 1$, set $$\bb_{n} = \cc^{\left\lceil n/2\right\rceil}+\sum_{i=1}^{\left\lceil n/2\right\rceil} \cc^{\left\lceil n/2\right\rceil-i}y^{i+1}.$$
 Thanks to the inequality $\left\lceil m/2\right\rceil+\left\lceil n/2\right\rceil \ge \left\lceil (m+n)/2\right\rceil$, it can be seen that $\bb_\bullet $ is a filtration of ideals in $S$. Note that $\alpha(\bb_{2n-1})=\alpha(\bb_{2n}) = \alpha(\cc^n) =n$ for all $n \in \mathbb{N}.$  Since $\alpha(I) =\alpha(\overline{I})$ for any homogeneous ideal $I$, we conclude that $\ahat(\bb_\bullet) = \dfrac{1}{2} = \ahat(\overline{\bb_\bullet})$.

 For any integer $s \ge 1$,
 \begin{equation}
  \label{eq_containment_ab}
  \aa_{s+1} =(x,y)^{s+1}= \sum_{i=0}^{s+1} \cc^{s+1-i}y^i \subseteq \cc^s+\sum_{i=2}^{s+1} \cc^{s+1-i}y^{i}=\bb_{2s}.
 \end{equation}
Observe also that $\aa_{s+1} \not\subseteq \bb_{2s+1}$ as $y^{s+1} \in \aa_{s+1} \setminus \bb_{2s+1}.$ Moreover,  $\aa_1 =(x,y) \not\subseteq \bb_1 =\overline{\bb_1}$. Therefore, $1 \le \rho(\aa_\bullet, \overline{\bb_\bullet}) \le \rho(\aa_\bullet,\bb_{\bullet}) = 1.$ Specifically, we have
$$ \rho(\aa_\bullet, \overline{\bb_\bullet}) = \rho(\aa_\bullet,\bb_{\bullet}) = 1.$$

On the other hand, it follows from \Cref{cor.boundsW} and \Cref{lem.estimate} that
$$\dfrac{1}{2} \leq \rhat(\aa_\bullet,\overline{\bb_\bullet}) \le \rhat(\aa_\bullet,{\bb_\bullet}).$$
Furthermore, using \eqref{eq_containment_ab} and \Cref{asp_res_bound}, we get
\[
 \rhat(\aa_\bullet,{\bb_\bullet}) \le \lim\limits_{s\to \infty} \frac{s+1}{2s}=\dfrac{1}{2}.
\]
Thus, $\rhat(\aa_\bullet,\overline{\bb_\bullet})=\rhat(\aa_\bullet,\bb_\bullet)=\dfrac{1}{2}.$
Particularly, $\rhat(\aa_\bullet, \bb_\bullet) = \rhat(\aa_\bullet, \overline{\bb_\bullet}) \not= \rho(\aa_\bullet, \overline{\bb_\bullet}).$
 \end{example}

By putting \Cref{thm.b1equivalent} and results in \Cref{sec.viaIntegralClosure} together, we obtain the following corollary, which generalizes \cite[Corollary 4.14]{DFMS}.

\begin{corollary}
	\label{cor.IntRes}
Let $S$ be a domain as in \Cref{asym_res_ideal}.
	Let $\aa_\bullet$ be a filtration of nonzero ideals and let $\bb_\bullet$ be a graded family of nonzero ideals in $S$. Suppose that $\bb_\bullet$ is $\bb$-equivalent for some ideal $\bb \subseteq S$. Then,
	$$\rhat(\aa_\bullet, \bb_\bullet) = \rhat(\aa_\bullet, \overline{\bb_\bullet}) = \rho(\aa_\bullet, \overline{\bb_\bullet}) = \max_{v \in \RV(\bb)} \left\{ \dfrac{v(\bb)}{\vhat(\aa_\bullet)} \right\} =\sup_{v(\bb) > 0} \left\{ \dfrac{v(\bb)}{\vhat(\aa_\bullet)} \right\} ,$$
	where the last supremum is taken over all valuations of $K$ supported on $S$ that take positive values in $\bb$.
\end{corollary}

\begin{proof} From \Cref{thm.ReesVal} (noting that the Rees algebra of the ideal $\bb$ is \emph{always} standard graded), we get
\[
\rhat(\aa_\bullet, \overline{\bb^\bullet}) = \max_{v \in \RV(\bb)} \left\{ \dfrac{v(\bb)}{\vhat(\aa_\bullet)} \right\} =\sup_{v(\bb) > 0} \left\{ \dfrac{v(\bb)}{\vhat(\aa_\bullet)} \right\}.
\]
Using \Cref{thm.b1equivalent}, we get
\[
\rhat(\aa_\bullet,\bb_\bullet) = \rhat(\aa_\bullet,\overline{\bb_\bullet}) = \rho(\aa_\bullet,\overline{\bb_\bullet})= \rhat(\aa_\bullet,\bb^\bullet)=\rho(\aa_\bullet, \overline{\bb^\bullet}).
\]
This gives the desired conclusion.
\end{proof}

%%%%%%%%%%%%%%%%%%%%%%%%%%%%%%%%%%%%%%%%%%%%%%%%%

In the remaining of this section, we shall present another instance where the conclusion of \Cref{thm.b1equivalent} holds for families that are not necessarily $\bb$-equivalent. We shall need a couple of auxiliary results.

\begin{lemma}\label{res_as_prod}
	Let $\aa_\bullet$ and $\bb_\bullet$ be graded families of ideals in $S$. Suppose that $\R^{[k]}(\bb_\bullet)$ is a standard graded $S$-algebra. Then, with $\bb_k^\bullet =\{ \bb_k^i\}_{i\geq 1}$, we have
	$$\rho(\aa_\bullet,\bb_k^\bullet) \le k\rho(\aa_\bullet,\bb_\bullet) \text{ and } \rho(\aa_\bullet,\overline{\bb_k^\bullet}) \le k\rho(\aa_\bullet, \overline{\bb_\bullet}).$$
\end{lemma}

\begin{proof}
We shall prove 	$\rho(\aa_\bullet,\bb_k^\bullet) \le k\rho(\aa_\bullet,\bb_\bullet)$; the other inequality follows by a similar line of arguments. Let $s,r \in \mathbb{N}$ be such that $\dfrac{s}{r} > k \rho(\aa_\bullet,\bb_\bullet).$ Then, $\dfrac{s}{kr} > \rho(\aa_\bullet,\bb_\bullet)$, which implies that  $\aa_s \subseteq \bb_{kr}$.  Since $\bb_{kr}= \bb_k^r$, we get $\aa_s \subseteq \bb_k^r$. Thus, $k\rho(\aa_\bullet,\bb_\bullet)$ is an upper bound of the set
	$$\left\{ \dfrac{s}{r} ~\Big|~ s,r \in \mathbb{N}, \text{ and } \aa_s \not\subseteq \bb_k^r \right\}.$$
	Hence, $\rho(\aa_\bullet,\bb_k^\bullet)\leq k\rho(\aa_\bullet,\bb_\bullet).$
\end{proof}

The inequality in \Cref{res_as_prod} can be strict in general, as depicted in the following example.

\begin{example} \label{ex.320strict}
	Let $S$ be a Noetherian polynomial ring over the field $\kk$, and let $\mm$ be its maximal homogeneous ideal. Let $I$ be a nonzero proper homogeneous ideal of $S$, and let $\nn$ be an ideal of $S$ such that $I^2\subseteq \nn \subseteq \mm I$.
	
	Consider graded families $\aa_\bullet$ and $\bb_\bullet$ of ideals in $S$, which are given by $\aa_i=I^i$ for $i \ge 1$, and
	\[
	\bb_i= \begin{cases}
		I^i &\text{if $i$ is even},\\
		\nn I^i, &\text{if $i$ is odd}.
	\end{cases}
	\]
	
	Note that $\R(\bb_\bullet)$ is Noetherian and generated in degree $1$ and $2$, and that $\bb_{2t}=\bb_2^t$ for all $t\ge 1$. We claim that  $\rho(\aa_\bullet,\bb_\bullet)=\rho(\aa_\bullet,\bb_2^\bullet)=2$. Particularly,
	\[
	\rho(\aa_\bullet,\bb_2^\bullet)=2 < 4= 2\rho(\aa_\bullet,\bb_\bullet).
	\]
	
	Indeed, for $s \in \NN$, set (with the convention that $\inf \varnothing = \infty$)
$$\beta_s(\aa_\bullet, \bb_\bullet) = \inf\{r ~\big|~ \aa_s \not\subseteq \bb_r\}.$$
Let $\NC(\aa_\bullet,\bb_\bullet)=\{(s, \beta_s(\aa_\bullet, \bb_\bullet)) ~\big|~ s \in \NN \text{ and } \beta_s(\aa_\bullet, \bb_\bullet) \not= \infty\}$. Clearly,
	\[
	\rho(\aa_\bullet,\bb_\bullet)=\sup \left\{\dfrac{s}{r} ~\Big|~ (s,r)\in \NC(\aa_\bullet,\bb_\bullet) \right\}.
	\]
	We shall first show that
	\[
	\NC(\aa_\bullet,\bb_2^\bullet)=\left\{ \left(s,\left\lceil \frac{s+1}{2} \right \rceil\right) ~\Big|~ s\ge 1\right\}.
	\]
	
	Take $(s,r)\in \NC(\aa_\bullet,\bb_2^\bullet)$. Then, $\aa_s=I^s\not\subseteq \bb_2^r=I^{2r}$, and so, $2r \ge s+1$, i.e., $r \ge \left\lceil \frac{s+1}{2} \right \rceil$. Since $\aa_s\subseteq \bb_2^j=I^{2j}$ for all $1\le j \le r-1$, we deduce that $r= \left\lceil \frac{s+1}{2} \right \rceil$, as desired.
	
	We shall next check that $(2,1) \in \NC(\aa_\bullet,\bb_\bullet)$. This holds since $\aa_2=I^2\not\subseteq \bb_1=\nn I$; otherwise, using $\mathfrak{n}\subseteq \mm I$, we get $I^2 \subseteq  \nn I \subseteq \mm I^2$, which is a contradiction. In fact, one can also easily show that
	\begin{gather*}
		\NC(\aa_\bullet,\bb_\bullet)= \{(s,s) ~\big|~ s\ge 1, s \, \text{is odd}\}  \,   \bigcup \, \{(s,s-1) ~\big|~ s\ge 2, s \, \text{is even}\}.
	\end{gather*}
	
	Now, we get
	\[
	\rho(\aa_\bullet,\bb_\bullet)=\max \left\{\sup \left\{\frac{s}{s} ~\big|~ s \ge 1, s \, \text{is odd}\right \}, \sup \left\{\frac{s}{s-1} ~\big|~ s \ge 1, s \, \text{is even}\right \} \right \}=2,
	\]
	and
	\[
	\rho(\aa_\bullet,\bb_2^\bullet)=\sup \left \{\frac{s}{\left\lceil \frac{s+1}{2} \right \rceil} ~\Big|~ s \ge 1 \right \}=2.
	\]
\end{example}

When the resurgence number $\rho(\bullet, \bullet)$ is replaced by the asymptotic resurgence number $\rhat(\bullet, \bullet)$, the inequality in \Cref{res_as_prod} becomes an equality.

\begin{lemma}
	\label{asym_res_as_prod}
	Let $\aa_\bullet$ and $\bb_\bullet$ be graded families of ideals in $S$. Suppose that $\R^{[k]}(\bb_\bullet)$ is a standard graded $S$-algebra. Then, with $\bb_k^\bullet =\{ \bb_k^i\}_{i\geq 1}$, we have
	$$\rhat(\aa_\bullet,\bb_k^\bullet)=k\rhat(\aa_\bullet,\bb_\bullet) \text{ and } \rhat(\aa_\bullet, \overline{\bb_k^\bullet})=k\rhat(\aa_\bullet, \overline{\bb_\bullet}).$$
\end{lemma}

\begin{proof} We shall prove $\rhat(\aa_\bullet,\bb_k^\bullet)=k\rhat(\aa_\bullet,\bb_\bullet)$; the other equality follows by \Cref{thm.ReesVal}. Let $s,r \in \mathbb{N}$ be such that $\dfrac{s}{r} > k \rhat(\aa_\bullet,\bb_\bullet).$ Then, $\dfrac{s}{kr} > \rhat(\aa_\bullet,\bb_\bullet)$, which implies that  $\aa_{st} \subseteq \bb_{krt}$ for infinitely many values $t$.  Since $\bb_{krt}= \bb_k^{rt}$, we get that $\aa_{st} \subseteq \bb_k^{rt}$ for infinitely many values $t$. Thus, $k\rhat(\aa_\bullet,\bb_\bullet)$ is an upper bound of the set
	$$\left\{ \dfrac{s}{r} ~\Big|~ s,r\in \mathbb{N}, \text{ and } \aa_{st} \not\subseteq \bb_k^{rt} \text{ for } t \gg 1\right\}.$$
	Hence, $\rhat(\aa_\bullet,\bb_k^\bullet)\leq k\rhat(\aa_\bullet,\bb_\bullet).$
	
	Now, consider $s,r \in \mathbb{N}$ be such that $\dfrac{s}{r} >  \rhat(\aa_\bullet,\bb_k^\bullet).$ Then,   $\aa_{st} \subseteq \bb_{k}^{rt}=\bb_{krt}$ for infinitely many values $t$, as $\bb_{krt}= \bb_k^{rt}$.  	This implies that
	$$\inf\left\{ \dfrac{s}{kr} ~\Big|~ s,r \in \mathbb{N}, \text{ and } \dfrac{s}{r} > \rhat(\aa_\bullet,\bb_k^\bullet) \right\}$$
	is an upper bound of the set
	$$\left\{ \dfrac{s}{r} ~\Big|~ s,r \in \mathbb{N}, \text{ and } \aa_{st} \not\subseteq \bb_{rt} \text{ for } t \gg 1\right\}.$$
	Note that
	$$\inf\left\{ \dfrac{s}{kr} ~\Big|~ s,r \in \mathbb{N}, \text{ and } \dfrac{s}{r} > \rhat(\aa_\bullet,\bb_k^\bullet) \right\}= \dfrac{1}{k}\inf\left\{ \dfrac{s}{r} ~\Big|~ s,r \in \mathbb{N}, \text{ and } \dfrac{s}{r} > \rhat(\aa_\bullet,\bb_k^\bullet) \right\}=\dfrac{\rhat(\aa_\bullet,\bb_k^\bullet)}{k}.$$
	Thus, $\rhat(\aa_\bullet,\bb_\bullet)\leq \dfrac{\rhat(\aa_\bullet,\bb_k^\bullet)}{k}.$ Hence, the assertion follows.
\end{proof}

We now arrive at the next main result of this section.

\begin{theorem}
	\label{thm.rhoInt}
Let $S$ be a domain as in \Cref{asym_res_ideal}. Let $\aa_\bullet$ be a filtration and let $\bb_\bullet$ be a graded family of nonzero ideals in $S$.  Suppose  that $\R^{[k]}(\bb_\bullet)$ is a standard graded $S$-algebra, and that $\rho(\aa_\bullet, \overline{\bb_k^\bullet}) = k\rho(\aa_\bullet, \overline{\bb_1^\bullet}).$ Then,
	$$\rhat(\aa_\bullet,\bb_\bullet) = \rhat(\aa_\bullet,\overline{\bb_\bullet}) = \rho(\aa_\bullet,\overline{\bb_\bullet}).$$
\end{theorem}

\begin{proof}
	By \Cref{lem.estimate} and \Cref{res_as_prod}, we have
	$$\rho(\aa_\bullet, \overline{\bb_k^\bullet}) \le k\rho(\aa_\bullet, \overline{\bb_\bullet}) \le k\rho(\aa_\bullet, \overline{\bb_1^\bullet}) = \rho(\aa_\bullet, \overline{\bb_k^\bullet}).$$
	Thus,
	$$\rho(\aa_\bullet,\overline{\bb_\bullet}) = \dfrac{ \rho(\aa_\bullet,\overline{\bb_k^\bullet}) }{k}.$$
	On the other hand, by \Cref{asym_res_as_prod}, we have
	$$\rhat(\aa_\bullet,{\bb_\bullet}) = \dfrac{ \rhat(\aa_\bullet,{\bb_k^\bullet})}{k} \text{ and } \rhat(\aa_\bullet,\overline{\bb_\bullet}) = \dfrac{ \rhat(\aa_\bullet,\overline{\bb_k^\bullet})}{k}.$$
	Thus, it is enough to prove that $\rho(\aa_\bullet,\overline{\bb_k^\bullet})  =\rhat(\aa_\bullet,\overline{\bb_k^\bullet})=\rhat(\aa_\bullet,{\bb_k^\bullet}) .$ This is indeed true, by  \Cref{asym_res_ideal}. The theorem is established.
\end{proof}

We end this section with the following broad question.

\begin{question} Classify pairs $(\aa_\bullet, \bb_\bullet)$ of graded families of ideals for which
$$\rhat(\aa_\bullet, \bb_\bullet) = \rhat(\aa_\bullet,\overline{\bb_\bullet}) = \rho(\aa_\bullet, \overline{\bb_\bullet}).$$
\end{question}

%%%%%%%%%%%%%%%%%%%%%%%%%%%%%%%%%%%%%%%%%%%%%%%%%%%%%

\section{Resurgence numbers as limits}
\label{sec.compute}

In this section, we construct sequences whose limits realize the asymptotic resurgence numbers $\rhat(\aa_\bullet, \bb_\bullet)$ and $\rhat(\aa_\bullet,\overline{\bb_\bullet})$, avoiding the use of Rees valuations as in Section \ref{sec.viaIntegralClosure}. Our construction is inspired by a recent work of DiPasquale, the fourth author, and Seceleanu \cite{DNS23}.  We also discuss another version of the resurgence number, denoted by $\rlim(\aa_\bullet, \bb_\bullet)$, and show that in many practical situations,
$$\rhat(\aa_\bullet, \bb_\bullet) = \rlim(\aa_\bullet, \bb_\bullet).$$
Our main results in this section are Theorems \ref{thm.limlambdaIJ}, \ref{thm.limbetarho}, and \ref{thm.rhatrlimrho}.

\begin{definition}
\label{def:subsuper}
A sequence of real  numbers $\{\alpha_n\}_{n\geq n_0}$ for some $n_0\in\NN$ is called
\begin{itemize}
\item \emph{sub-additive} if $\alpha_{i+j} \leq \alpha_i+\alpha_j$ for all $i,j\geq n_0$; and
\item \emph{super-additive} if $\alpha_{i+j} \geq \alpha_i+\alpha_j$ for all $i,j\geq n_0$.
\end{itemize}
\end{definition}

 Fekete's Lemma guarantees the existence of the limit
 $\widehat{\alpha}:=\lim_{n\to\infty} \frac{\alpha_n}{n}$
 for any sub-additive  or super-additive sequence of real numbers $\{\alpha_n\}_{n\ge n_0}$, where the limit may be $\pm \infty$. For a sub-additive sequence $\{\alpha_n\}_{n \ge n_0}$, we have
 $$\widehat{\alpha} =  \lim_{n \to \infty} \dfrac{\alpha_n}{n} = \inf_{n\ge  n_0} \frac{\alpha_n}{n}.$$
 For a super-additive sequence $\{\alpha_n\}_{n \ge n_0}$, we have
 $$\widehat{\alpha} =  \lim_{n \to \infty} \dfrac{\alpha_n}{n} = \sup_{n\ge n_0} \frac{\alpha_n}{n}.$$

    \begin{definition}[{\cite[Definition 2.1]{DNS23}}]
    \label{def.dualseq}
        Given two sequences $\{\alpha_n\}_{n \ge n_0}$ and $\{\beta_n\}_{n \ge m_0}$ of real numbers, we define two new sequences $\{\overleftarrow{\alpha}^\beta_n\}_{n \ge m_0}$ and $\{\overrightarrow{\alpha}^\beta_n\}_{n \ge m_0}$, that are given by
        \begin{eqnarray*}
        \overleftarrow{\alpha}^\beta_n=\inf\{d \ge n_0 \mid \alpha_d\geq \beta_n\},\\
        \overrightarrow{\alpha}^\beta_n=\sup\{d \ge n_0 \mid \alpha_d\leq \beta_n\}.
        \end{eqnarray*}
    \end{definition}

It is known that the sequences $\{\alpha_n\}_{n \ge n_0}$, $\{\beta_n\}_{n \ge m_0}$, $\{\overleftarrow{\alpha}^\beta_n\}_{n \ge m_0}$, and $\{\overrightarrow{\alpha}^\beta_n\}_{n \ge m_0}$ have the following relationship. We modify the statement to fit our purpose but the proof is the same as that of \cite[Theorem 2.5]{DNS23}.

    \begin{theorem}[{\cite[Theorem 2.5]{DNS23}}]
    \label{thm.dualseq}
        Suppose that $\{\alpha_n\}_{n \ge n_0}$ and $\{\beta_n\}_{n \ge m_0}$ are sub-additive and super-additive sequences of positive real numbers, respectively.
        \begin{enumerate}
            \item Assume that, for $n\gg 0$, $\overrightarrow{\alpha}^\beta_n \in \NN$. Then, the sequence $\{\overrightarrow{\alpha}^\beta_n\}$ is super-additive and $\widehat{\overrightarrow{\alpha}^\beta}=\frac{\widehat{\beta}}{\widehat{\alpha}}$.
            \item Assume that, for $n\gg 0$, $\overleftarrow{\beta}^\alpha_n \in \NN$. Then, the sequence $\{\overleftarrow{\beta}^\alpha_n\}$ is sub-additive and $\widehat{\overleftarrow{\beta}^\alpha}=\frac{\widehat{\alpha}}{\widehat{\beta}}$.
        \end{enumerate}
    \end{theorem}

\begin{definition} \label{def.lambdabeta}
	Let $S$ be a domain and let $v$ be a valuation of $K = \text{QF}(S)$. For $n \in \NN$, set
$$\lambda_n(\aa_\bullet,\bb_\bullet):=\sup \{ d \mid \aa_d \not \subseteq \bb_n\} \text{ and } \lambda_n^v(\aa_\bullet,\bb_\bullet):= \sup \{d \mid v(\aa_d)<v(\bb_n) \}.$$
\end{definition}

Note that $\lambda_n^v(\aa_\bullet,\bb_\bullet)$ and $\lambda_n(\aa_\bullet,\bb_\bullet)$ in general can take infinite values. Before stating our first main result of this section, we have the following useful proposition.

\begin{proposition}
    \label{prop.lambdavalIJ}
    Let $S$ be a domain, and let $\aa_\bullet$ and $\bb_\bullet$ be graded families of nonzero ideals in $S$.  Suppose that $\R^{[k]}(\bb_\bullet)$ is a standard graded $S$-algebra, for some $k \in \NN$. For any valuation $v$ of $K$ with $v(\bb_k)>0$ and $\widehat{v}(\aa_\bullet) >0$, we have
    \begin{enumerate}
        \item $\{\lambda_{kn}^v(\aa_\bullet,\bb_\bullet)\}$ is a super-additive sequence.
        \item $\displaystyle \widehat{\lambda_{kn}^v(\aa_\bullet,\bb_\bullet)}=\lim_{n\rightarrow \infty} \frac{\lambda_{kn}^v(\aa_\bullet,\bb_\bullet)}{n}= \frac{v(\bb_k)}{\widehat{v}(\aa_\bullet)}$.
    \end{enumerate}
\end{proposition}

\begin{proof}
    We first give an alternate definition for $\lambda^v_{kn}(\aa_\bullet,\bb_\bullet)$ by means of Definition \ref{def.dualseq}. Set $\alpha_n=v(\aa_n)$ and $\delta_n=nv(\bb_k)-1$ for all $n$. Note that $\{\alpha_n\}_{n \in \NN}$ is sub-additive with $\widehat{\alpha}=\widehat{v}(\aa_\bullet)$ and $\{\delta_{n}\}_{n \in \NN}$ is super-additive with $\widehat{\delta}={v(\bb_k)}$, by \Cref{supported}. By definition, one see that $\lambda^v_{kn}(\aa_\bullet,\bb_\bullet)=\overrightarrow{\alpha}^\delta_{n}$ for every $n \in \NN$.

    Also, as $\widehat{v}(\aa_\bullet) >0$, we have $v(a_1)>0$. Since $v(\bb_k)>0$, there exists $n_0 \in \NN$ such that $nv(b_k) \ge n_0 v(\bb_k) > v(\aa_1)$ for all $n \ge n_0$. Therefore, for $n \ge n_0$, the set $\{d \mid v(\aa_d)<v(\bb_{kn}) \}$ is a nonempty set.  Also,  for each $n \ge n_0$, the set $\{d \mid v(\aa_d)<v(\bb_{kn}) \}$ is finite, as if $v(\aa_d)< v(\bb_{kn})$ for infinitely many $d$, then $\vhat(\aa_\bullet)=0$. Hence, $\overrightarrow{\alpha}_n^{\delta} \in \NN$ for $n\ge n_0$, and the result follows from Theorem \ref{thm.dualseq}.
\end{proof}

It would be desirable to know when the sequence ${\displaystyle \left\{\lambda_n^v(\aa_\bullet,\bb_\bullet)/n \right\}_{n \in \NN}}$ has a limit.

\begin{question}
For which graded families $\aa_\bullet, \bb_\bullet$, does  $\displaystyle \lim_{n\rightarrow \infty} \frac{\lambda_{n}^v(\aa_\bullet,\bb_\bullet)}{n}$ exist?
\end{question}

The following example shows that, in general, the sequence $\left\{\lambda_n^v(\aa_\bullet, \bb_\bullet)/n \right\}_{n \in \NN}$ may have distinct subsequences converging to different limits, even when $\aa_\bullet = \bb_\bullet$ are graded filtration.

\begin{example} \label{ex.lambdaLim}
	Let $S=\kk[x]$, $\mm=(x)$, and let $v$ be the $\mm$-adic valuation. Consider the sequence  $\aa_\bullet =\{\aa_n\}_{n \in \NN}$ of ideals in $S$, given by
	\[
	\aa_n=\mm^{\lceil \log_2(n+1) \rceil}.
	\]
It can be seen that $\aa_\bullet$ is a filtration of ideals. This is because $n \mapsto \lceil \log_2(n+1) \rceil$ is a sub-additive function. We shall show that $$\lambda_n^v(\aa_\bullet, \aa_\bullet) = 2^{\lceil \log_2(n+1)\rceil -1}-1.$$
Indeed, by definition, we have
	\[
	\lambda^v_n(\aa_\bullet,\aa_\bullet)=\sup\{d \mid v(\aa_d)<v(\aa_n)\} = \sup\{d \mid \lceil \log_2(d+1) \rceil < \lceil \log_2(n+1) \rceil\}.
	\]
Set $t= \lceil \log_2(n+1) \rceil$. Then $d =\lambda^v_n(\aa_\bullet,\aa_\bullet)$ is the largest integer satisfying
	  $\lceil \log_2(d+1) \rceil < t;$
that is, $\log_2(d+1) \le t-1$ or, equivalently, $d \le 2^{t-1}-1$.  Therefore, $\lambda_n^v(\aa_\bullet, \aa_\bullet) = 2^{t-1}-1$, as claimed.

For $n=2^s-1$, where $s\in \ZZ_{> 0}$, we have $\lceil \log_2(n+1) \rceil=s$. In this case,
	\[
	\frac{\lambda^v_n(\aa_\bullet,\aa_\bullet)}{n} = \frac{2^{s-1}-1}{2^s-1} \xrightarrow{s\to \infty} \frac{1}{2}.
	\]
On the other hand, for $n=2^s$, where $s\in \ZZ_{>0}$, we have $\lceil \log_2(n+1) \rceil=s+1$. In this case,
	\[
	\frac{\lambda^v_n(\aa_\bullet,\aa_\bullet)}{n} = \frac{2^{s}-1}{2^s} \xrightarrow{s \rightarrow \infty} 1.
	\]
Hence, $\left\{\lambda_n^v(\aa_\bullet, \aa_\bullet)/n \right\}_{n \in \NN}$ has two subsequences with limits $\frac{1}{2}$ and 1, respectively. Therefore, $\lim\limits_{n\rightarrow \infty} \lambda_n^v(\aa_\bullet, \aa_\bullet)/n$ does not exist.
\end{example}
We proceed to our first main result of this section, which shows that in certain situation, $\widehat{\rho}(\aa_\bullet,\overline{\bb_\bullet})$ can be computed as limits of the (subsequence of) $\lambda_n/n$ sequence.
\begin{theorem}
    \label{thm.limlambdaIJ}
      Let $S$ be a domain, and let $\aa_\bullet$ and $\bb_\bullet$ be graded families of nonzero ideals in $S$. Assume that  $\R^{[k]}(\bb_\bullet)$ is a standard graded $S$-algebra, for some $k \in \NN$.  For $n \ge 1$, set $\lambda_n=\lambda_n(\aa_\bullet,\bb_\bullet)$, $\overline{\lambda_n}=\lambda_n(\aa_\bullet,\overline{\bb_\bullet})$, and for a valuation $v$ of $K$, set  $\lambda_n^v=\lambda_n^v(\aa_\bullet,\bb_\bullet)$. If $\rhat(\aa_\bullet, \overline{\bb_\bullet}) < \infty$, then there exists a valuation $v_0$ (which can be chosen as a Rees valuation of $\bb_k$) such that
    $$\widehat{\rho}(\aa_\bullet,\overline{\bb_\bullet}) =\lim_{n\rightarrow \infty}  \frac{\lambda_{kn}^{v_0}}{kn}.$$
Furthermore, if $\R(\overline{\bb_k^\bullet})$ is a finitely generated $\R(\bb_k^\bullet)$-module, then
    $$\widehat{\rho}(\aa_\bullet,\overline{\bb_\bullet}) = \lim_{n\rightarrow \infty} \frac{\lambda_{kn}}{kn}  = \lim_{n\rightarrow \infty} \frac{\overline{\lambda_{kn}}}{kn}  = \lim_{n\rightarrow \infty}  \frac{\lambda_{kn}^{v_0}}{kn}.$$
\end{theorem}

\begin{proof}
By Lemma \ref{asym_res_as_prod}, with $\bb_k^\bullet =\{ \bb_k^i\}_{i\geq 1}$, we have
	$ \rhat(\aa_\bullet, \overline{\bb_k^\bullet})=k\rhat(\aa_\bullet, \overline{\bb_\bullet}).$ Since $\widehat{\rho}(\aa_\bullet,\overline{\bb_\bullet})$ is finite, we have $\rhat(\aa_\bullet, \overline{\bb_k^\bullet})<\infty$, and therefore,  by Theorem \ref{thm.ReesVal} and Proposition \ref{prop.lambdavalIJ}, there exists a valuation $v_0 \in \RV(\bb_k)$  such that $\widehat{v_0}(\aa_\bullet) > 0$ and
    $$k\widehat{\rho}(\aa_\bullet,\overline{\bb_\bullet}) =\rhat(\aa_\bullet, \overline{\bb_k^\bullet}) =  \lim_{n\rightarrow \infty}  \frac{\lambda_{kn}^{v_0}}{n} .$$
  For the latter claim, note that  for any valuation $v$ of $K$ with $v(\bb_k)>0$ and for every $n \ge 1$,
    \[
    \{d \mid  v(\aa_d)<v(\bb_n) \} \subseteq \{d \mid \aa_d \not \subseteq \overline{\bb_n} \} \subseteq \{d \mid \aa_d \not \subseteq \bb_n \}.
    \]
  Therefore, $\lambda_n^v\le \overline{\lambda_n} \le \lambda_n$ for every $n$. We shall show that the with the valuation $v_0$ above, we have
    $$\lim_{n\rightarrow \infty} \frac{\lambda_{kn}}{n}  = \lim_{n\rightarrow \infty} \frac{\overline{\lambda_{kn}}}{n}  = \lim_{n\rightarrow \infty}  \frac{\lambda_{kn}^{v_0}}{n} = {\rhat(\aa_\bullet, \overline{\bb_k^\bullet})}.$$

Indeed, since $\R(\overline{\bb_k^\bullet})$ is a finitely generated $\R(\bb_k^\bullet)$-module, there exists a positive integer $m$ such that  $ \overline{\bb_k^n} \subseteq \bb_k^{n-m}$ for all $n \ge m$.
  Therefore, $\aa_{\lambda_{kn}} \not \subseteq \bb_{kn}=\bb_k^n$ implies that $\aa_{\lambda_{kn}} \not \subseteq \overline{\bb_k^{n+m}}$ for every $n \in \NN$. By Lemma \ref{tech_lemm} (1), we have $\dfrac{\lambda_{kn}}{n+m}<\rhat(\aa_\bullet, \overline{\bb_k^\bullet})$.

   Thus,
    \[
    \rhat(\aa_\bullet, \overline{\bb_k^\bullet}) = \liminf_{n\rightarrow \infty} \frac{\lambda_{kn}^{v_0}}{n+m} \le  \liminf_{n\rightarrow \infty} \frac{\overline{\lambda_{kn}}}{n+m}  \le \liminf_{n\rightarrow \infty} \frac{\lambda_{kn}}{n+m}
    \]
    and
    \[
    \rhat(\aa_\bullet, \overline{\bb_k^\bullet}) = \limsup_{n\rightarrow \infty} \frac{\lambda_{kn}^{v_0}}{n+m} \le  \limsup_{n\rightarrow \infty} \frac{\overline{\lambda_{kn}}}{n+m}  \le \limsup_{n\rightarrow \infty} \frac{\lambda_{kn}}{n+m}.
    \]
    The fact that $\dfrac{\lambda_{kn}}{n+m}<\rhat(\aa_\bullet, \overline{\bb_k^\bullet})$ for every $n \in \NN$ implies that $\displaystyle \limsup_{n\rightarrow \infty} \frac{\lambda_{kn}}{n+m} \le \rhat(\aa_\bullet, \overline{\bb_k^\bullet})$, which shows that all limit supremum and limit infimum above exist and equal. Hence,
    $$\lim_{n\rightarrow \infty} \frac{\lambda_{kn}}{n}  = \lim_{n\rightarrow \infty} \frac{\overline{\lambda_{kn}}}{n}  = \lim_{n\rightarrow \infty}  \frac{\lambda_{kn}^{v_0}}{n} = \rhat(\aa_\bullet, \overline{\bb_k^\bullet})$$
    as desired.
\end{proof}

Dual to $\{\lambda_n(\aa_\bullet,\bb_\bullet)\}_{n \in \NN}$ is the sequence $\{\beta_n(\aa_\bullet, \bb_\bullet)\}_{n \in \NN}$, which was already used in \Cref{ex.320strict}, namely,
\[
    \beta_n(\aa_\bullet,\bb_\bullet):=\inf \{ d \mid \aa_n \not \subseteq \bb_d\}.
\]
Also, for a valuation $v$ of $K$, set
$$\beta_n^v(\aa_\bullet,\bb_\bullet):= \inf \{d \mid v(\aa_n)<v(\bb_d) \}.$$

In the case where $\bb_\bullet = \bb^\bullet$, the following result is a dual version of \Cref{prop.lambdavalIJ}.

\begin{proposition}
    \label{prop.betavalIJ}
    Let $S$ be a domain and let $\aa_\bullet$ be a graded family of nonzero ideals in $S$. Let $\bb \subseteq S$ be an ideal and, as usual, set $\bb^\bullet = \{\bb^i\}_{i \ge 1}$. For any valuation $v$ of $K$ with $v(\bb)>0$, we have
    \begin{enumerate}
        \item $\beta_{n}^v(\aa_\bullet,\bb^\bullet)$ is a sub-additive sequence; and
        \item $\displaystyle \widehat{\beta_{n}^v(\aa_\bullet,\bb^\bullet)}=\lim_{n\rightarrow \infty} \frac{\beta_{n}^v(\aa_\bullet,\bb^\bullet)}{n}=\inf_{n \in \NN} \left\{ \frac{\beta_{n}^v(\aa_\bullet,\bb^\bullet)}{n}\right\} = \frac{\widehat{v}(\aa_\bullet)}{\widehat{v}(\bb^\bullet)}=\frac{\widehat{v}(\aa_\bullet)}{v(\bb)}$.
    \end{enumerate}
\end{proposition}

\begin{proof}
    Set $\alpha_n=v(\aa_{n})$ and $\delta_n=nv(\bb)-1$ for all $n \in \NN.$ Note that $\{\alpha_n\}_{n \in \NN}$ is sub-additive with $\widehat{\alpha}=\widehat{v}(\aa_\bullet)$ and $\{\delta_{n}\}_{n \in \NN}$ is super-additive with $\widehat{\delta}=v(\bb)=\widehat{v}(\bb^\bullet)$.
    By definition, it can be seen that $\beta^v_{n}(\aa_\bullet,\bb_\bullet)=\overleftarrow{\delta}^\alpha_{n}$ for every $n \in \NN$.

    Since $v(\bb)>0$, the set $\{d \mid v(\aa_n)<v(\bb^{d}) \}$ is non-empty for all $n \in \NN$. Thus, $\overleftarrow{\delta}_n^{\alpha} \in \NN$ for $n \in \NN$, and the result follows from Theorem \ref{thm.dualseq}.
\end{proof}

\begin{question}
For which graded families $\aa_\bullet, \bb_\bullet$, does $\displaystyle \lim_{n\rightarrow \infty} \frac{\beta_{n}^v(\aa_\bullet,\bb_\bullet)}{n}$ exist?
\end{question}

Example \ref{ex.lambdaLim} also gives an instance where the sequence $\left\{\beta^v_n(\aa_\bullet,\bb_\bullet)/n \right\}_{n \in \NN}$ may have distinct subsequences converging to different limits.

\begin{example}
Let $S$ and $\aa_\bullet = \{\mm^{\lceil \log_2(n+1)\rceil}\}_{n \in \NN}$ be as in \Cref{ex.lambdaLim}.
As we have seen in \Cref{ex.lambdaLim}, $\aa_\bullet$ is a graded filtration of ideals.
Moreover, by a similar argument, it can also be shown that
$$\beta^v_n(\aa_\bullet,\aa_\bullet)=2^{\lceil \log_2(n+1) \rceil}.$$

Consider $n=2^s-1$, where $s\in \ZZ_{>0}$. Then,  $\lceil \log_2(n+1) \rceil=s$, and so
	\[
	\frac{\beta^v_n(\aa_\bullet,\aa_\bullet)}{n} = \frac{2^s}{2^s-1} \xrightarrow{s\to \infty} 1.
	\]
On the other hand, consider $n=2^s$, where $s\in \mathbb{Z}_{>0}$. Then,  $\lceil \log_2(n+1) \rceil=s+1$, and so
	\[
	\frac{\beta^v_n(\aa_\bullet,\aa_\bullet)}{n} = \frac{2^{s+1}}{2^s} = 2.
	\]
Hence, $\left\{\beta^v_n(\aa_\bullet,\bb_\bullet)/n \right\}_{n \in \NN}$ has two subsequences converging to 1 and 2, respectively. Therefore, $\lim\limits_{n\rightarrow \infty} \beta_n^v(\aa_\bullet, \aa_\bullet)/n$ does not exist.
\end{example}

We are ready to prove our next main result, where $\rhat(\aa_\bullet, \overline{\bb_\bullet})$ can be realized as the reciprocal of the limit of the $\beta_n(\aa_\bullet, \bb_\bullet)/n$ sequence. Using this $\beta_n$ sequence, we can also slightly improve \Cref{thm.limlambdaIJ} by not having to require that $\rhat(\aa_\bullet, \overline{\bb_\bullet}) < \infty$.

\begin{theorem} \label{thm.limbetarho}
Let $S$ be a domain, let $\aa_\bullet$ be a graded family of ideals, and let $\bb_\bullet$ be a filtration of ideals in $S$. For $n \ge 1$, set $\beta_n=\beta_n(\aa_\bullet,\bb_\bullet)$, $\overline{\beta_n}=\beta_n(\aa_\bullet,\overline{\bb_\bullet})$, and for any valuation $v$ of $K$, set $\beta_n^v=\beta_n^v(\aa_\bullet,\bb_\bullet)$. Suppose that $\R^{[k]}(\bb_\bullet)$ is a standard graded $S$-algebra and $\R(\overline{\bb_k^{\bullet}})$ is a finitely generated $\R(\bb_k^{\bullet})$-module, for some $k \in \NN$. Then, there exists a valuation $v_0$ (which can be chosen as a Rees valuation of $\bb_k$) such that
    $$\frac{1}{\widehat{\rho}(\aa_\bullet,\overline{\bb_\bullet})} = \lim_{n\rightarrow \infty} \frac{\beta_{n}}{n}  = \lim_{n\rightarrow \infty} \frac{\overline{\beta_{n}}}{n}  = \lim_{n\rightarrow \infty}  \frac{\beta_{n}^{v_0}}{n}.$$
\end{theorem}

\begin{proof}
    For any valuation $v$ of $K$ with $v(\bb_k)>0$ and for every $n \in \NN$, we have
    \[
    \{d \mid  v(\aa_n)<v(\bb_d) \} \subseteq \{d \mid \aa_n \not \subseteq \overline{\bb_d} \} \subseteq \{d \mid \aa_n \not \subseteq \bb_d \}.
    \]
   This implies that $\beta_n^v\ge \overline{\beta_n} \ge \beta_n$ for $n \in \NN$.

   Since $\R(\overline{\bb_k^{\bullet}})$ is a finitely generated $\R(\bb_k^{\bullet})$-module, there exists a positive integer $m$ such that  $\overline{\bb_k^n} \subseteq \bb_k^{n-m}$ for all $n \ge m.$
   By \Cref{thm.ReesVal} and Proposition \ref{prop.betavalIJ}, there exists a Rees valuation $v_0$ of $\bb_k$ such that
    $$\frac{1}{\rhat(\aa_\bullet, \overline{\bb_k^\bullet})} =  \frac{\vhat_{0}(\aa_\bullet)}{v_0(\bb_k)}= \lim_{n \to \infty}  \frac{\beta_{n}^{v_0}(\aa_\bullet,\bb_k^{\bullet})}{n} = \lim_{n \to \infty}  \frac{\beta_{n}^{v_0}(\aa_\bullet,\bb_k^{\bullet})+m}{n}.$$
    Therefore,
    \[
    \frac{1}{\widehat{\rho}(\aa_\bullet,\overline{\bb_k^\bullet})} = \liminf_{n\rightarrow \infty} \frac{\beta_{n}^{v_0}(\aa_\bullet,\bb_k^{\bullet})+m}{n} \ge  \liminf_{n\rightarrow \infty} \frac{\overline{\beta_{n}}(\aa_\bullet,\bb_k^{\bullet})+m}{n}  \ge \liminf_{n\rightarrow \infty} \frac{\beta_{n}(\aa_\bullet,\bb_k^{\bullet})+m}{n}
    \]
    and
    \[
    \frac{1}{\widehat{\rho}(\aa_\bullet,\overline{\bb_k^\bullet})} = \limsup_{n\rightarrow \infty} \frac{\beta_{n}^{v_0}(\aa_\bullet,\bb_k^{\bullet})+m}{n} \ge  \limsup_{n\rightarrow \infty} \frac{\overline{\beta_{n}}(\aa_\bullet,\bb_k^{\bullet})+m}{n}  \ge \limsup_{n\rightarrow \infty} \frac{\beta_{n}(\aa_\bullet,\bb_k^{\bullet})+m}{n}.
    \]

    As $\aa_n\not \subseteq \bb_k^{\beta_n(\aa_\bullet,\bb_k^{\bullet})}$, we have $\aa_n \not \subseteq \overline{\bb_k^{\beta_n(\aa_\bullet,\bb_k^{\bullet})+m}}$ for all $n \in \NN$, and so, $\displaystyle \frac{n}{\beta_n(\aa_\bullet,\bb_k^{\bullet})+m}<\widehat{\rho}(\aa_\bullet,\overline{\bb_k^\bullet})$ by Lemma \ref{tech_lemm} (1).
     Consequently, $\displaystyle \frac{1}{\widehat{\rho}(\aa_\bullet, \overline{\bb_k^\bullet})}\le \liminf_{n \to \infty} \frac{\beta_n(\aa_\bullet,\bb_k^{\bullet})+m}{n}$. It follows that
     $$\lim_{n\rightarrow \infty} \frac{\beta_{n}(\aa_\bullet,\bb_k^{\bullet})}{n}  = \lim_{n\rightarrow \infty} \frac{\overline{\beta_{n}}(\aa_\bullet,\bb_k^{\bullet})}{n}  = \lim_{n\rightarrow \infty}  \frac{\beta_{n}^{v_0}(\aa_\bullet,\bb_k^{\bullet})}{n} = \frac{1}{\widehat{\rho}(\aa_\bullet,\overline{\bb_k^\bullet})}.$$

     Next, we claim that for $n \in \NN$,
     \begin{align}
    k\left( \beta_{n}^{v_0}(\aa_\bullet,\bb_k^{\bullet})-1\right) &\le \beta_{n}^{v_0}(\aa_\bullet,\bb_{\bullet}) \le k \beta_{n}^{v_0}(\aa_\bullet,\bb_k^{\bullet}), \label{ineq_1stchain}\\
     k\left( \overline{\beta_{n}}(\aa_\bullet,\bb_k^{\bullet})-1\right) &\le \overline{\beta_{n}}(\aa_\bullet,\bb_{\bullet}) \le k \overline{\beta_{n}}(\aa_\bullet,\bb_k^{\bullet}), \label{ineq_2ndchain}\\
     k\left( \beta_{n}(\aa_\bullet,\bb_k^{\bullet})-1\right) &\le \beta_{n}(\aa_\bullet,\bb_{\bullet}) \le k \beta_{n}(\aa_\bullet,\bb_k^{\bullet}). \label{ineq_3rdchain}
     \end{align}

For each of the inequalities on the left of the last three chains, we need the hypothesis that $\bb_\bullet$ is a filtration. For clarity, we prove \eqref{ineq_2ndchain}, similar arguments work for the remaining chains.

For the inequality on the left of \eqref{ineq_2ndchain}, it is harmless to assume that $\overline{\beta_n}(\aa_\bullet,\bb_\bullet)<\infty$. Per definition, $\overline{\beta_n}(\aa_\bullet,\bb_\bullet)=\inf\{d: \aa_n\not\subseteq \overline{\bb_d}\} <\infty$, so as $\bb_\bullet$ is a filtration, we get the finiteness of
\[
\overline{\beta_n}(\aa_\bullet,\bb_k^\bullet)=\beta_n(\aa_\bullet,\overline{\bb_k^\bullet})=\inf\{d: \aa_n\not\subseteq \overline{\bb_k^d}\}=\inf\{d: \aa_n\not\subseteq \overline{\bb_{kd}}\}.
\]
The last display yields
$$
\aa_n \subseteq \overline{\bb_{k({\overline{\beta_n}(\aa_\bullet,\bb_k^\bullet)-1)}}}.
$$
Since $\bb_\bullet$ is a filtration, $\aa_n\subseteq \overline{\bb_d}$ for all $d\le k({\overline{\beta_n}(\aa_\bullet,\bb_k^\bullet)-1)}$. Hence $k({\overline{\beta_n}(\aa_\bullet,\bb_k^\bullet)-1)} <\overline{\beta_n}(\aa_\bullet,\bb_\bullet)$. This proves the inequality on the left of \eqref{ineq_2ndchain}.

For the inequality on the right, again it is harmless to assume that $\overline{\beta_n}(\aa_\bullet,\bb_k^\bullet)<\infty$. By definition, $\aa_n \not\subseteq \overline{\bb_k^{\overline{\beta_n}(\aa_\bullet,\bb_k^\bullet)}}=\overline{\bb_{k{\overline{\beta_n}(\aa_\bullet,\bb_k^\bullet)}}}$. This yields $\overline{\beta_n}(\aa_\bullet,\bb_\bullet)\le k{\overline{\beta_n}(\aa_\bullet,\bb_k^\bullet)}$, as claimed.

Now, applying the Sandwich theorem for limits for \eqref{ineq_1stchain} --  \eqref{ineq_3rdchain}, we get $$\lim_{n\rightarrow \infty} \frac{\beta_{n}}{n}  = \lim_{n\rightarrow \infty} \frac{\overline{\beta_{n}}}{n}  = \lim_{n\rightarrow \infty}  \frac{\beta_{n}^{v_0}}{n} = \frac{k}{\widehat{\rho}(\aa_\bullet,\overline{\bb_k^\bullet})}=\frac{1}{\widehat{\rho}(\aa_\bullet,\overline{\bb_\bullet})},$$
     where the last equality holds by \Cref{asym_res_as_prod}. Hence, the assertion follows.
     \end{proof}

In the remaining of this section, we will focus on yet another version of resurgence, whose definition is motivated by \cite[Theorem 2.1 and Lemma 2.2]{BHJT} and \cite[Theorem 2.1]{HKZ}. This new version of resurgence arises as an actual limit of a well-constructed sequence, and is equal to the asymptotic resurgence number in practical situations; see \Cref{thm.rhatrlimrho}.

\begin{definition} \label{def.rlim}
	Let $\aa_\bullet$ and $\bb_\bullet$ be graded families of ideals in $S$.
	\begin{enumerate}
		\item Define a sequence $\{\rho^n(\aa_\bullet, \bb_\bullet)\}_{n \in \NN}$ as follows:
$$\rho^n(\aa_\bullet, \bb_\bullet) = \sup \left\{ \dfrac{s}{\beta_s(\aa_\bullet, \bb_\bullet)} ~\Big|~  \beta_s(\aa_\bullet, \bb_\bullet)< \infty \text{ and } s \ge n \right\}.$$
		\item Set
		$$\rlim(\aa_\bullet, \bb_\bullet) = \lim_{n \rightarrow \infty} \rho^n(\aa_\bullet, \bb_\bullet).$$
	\end{enumerate}
\end{definition}

Note that, in general, $\rho^n(\aa_\bullet, \bb_\bullet)$ can take infinite values.
Clearly, $\left\{\rho^n(\aa_\bullet, \bb_\bullet)\right\}_{n \ge 1}$ is a nonincreasing sequence, so it has a limit. That is, $\rlim(\aa_\bullet, \bb_\bullet)$ is well-defined.
Observe further that $\rlim(\aa_\bullet, \bb_\bullet) \le \rho^n(\aa_\bullet, \bb_\bullet)$ for all $n \ge 1$ and, by definition,
\begin{align}
\rho(\aa_\bullet, \bb_\bullet) = \sup_{s \in \NN} \left\{ \dfrac{s}{\beta_s(\aa_\bullet, \bb_\bullet)}~\Big|~ \beta_s(\aa_\bullet, \bb_\bullet)< \infty \text{ and } s \in \NN \right\} = \sup_{n \in \NN} \left\{ \rho^n(\aa_\bullet, \bb_\bullet)\right\}.\label{eq.rhosup}
\end{align}

It is easy to see that $\rhat(\aa_\bullet, \bb_\bullet) \le \rho^n(\aa_\bullet, \bb_\bullet) \le \rho(\aa_\bullet, \bb_\bullet)$ for all $n \ge 1$. Therefore,
$$\rhat(\aa_\bullet, \bb_\bullet) \le \rho^{\lim}(\aa_\bullet, \bb_\bullet) \le \rho(\aa_\bullet, \bb_\bullet).$$
In general, these inequalities can be strict as demonstrated in the following example.

\begin{example}
Let $I$ be a nonzero proper normal ideal in a Noetherian domain $S$. Consider  $\aa_\bullet$ and $\bb_\bullet $ with
$$\aa_i =I^i \text{ and } \bb_i = I^{\lceil{\sqrt i}\rceil} \text{ for all } i \ge 1.$$
As we have seen in \Cref{ex.low}.(3), $\rhat(\aa_\bullet, \bb_\bullet) = - \infty$ and $\rho(\aa_\bullet, \bb_\bullet) =\dfrac{1}{2}.$

We now compute $\rho^{\lim}(\aa_\bullet, \bb_\bullet)$. Observe that, if $r < s^2+1$ then $r \le s^2$, which implies that $\lceil \sqrt{r} \rceil \le s$, and so $\aa_s \subseteq \bb_r$. On the other hand, if $r = s^2+1$, then $\lceil \sqrt{r} \rceil = s+1$ and, therefore, $\aa_s \not\subseteq \bb_r$. Thus, for all $s \in \NN$, $\beta_s(\aa_\bullet,\bb_\bullet)=s^2+1.$

It is easily seen that $\left \{ \dfrac{s}{s^2+1} \right \}$ is a nonincreasing sequence. Thus, for every $n \in \NN$, $\rho^n(\aa_\bullet, \bb_\bullet) =\dfrac{n}{n^2+1}.$ Hence, $\rho^{\lim}(\aa_\bullet, \bb_\bullet) =0.$ Particularly, $\rhat(\aa_\bullet, \bb_\bullet) < \rho^{\lim}(\aa_\bullet, \bb_\bullet) < \rho(\aa_\bullet, \bb_\bullet)$.
\end{example}

The following results provide equalities involving $\rlim$ and are also useful for investigating the rationality of resurgence numbers as we will see in the last section.
\begin{theorem}\label{lem.rlim}
Let $\aa_\bullet, \bb_\bullet$ and $\bb'_\bullet$ be graded families of ideals in $S$. Suppose that $\bb_\bullet \le \bb'_\bullet$ and for some positive integer $k$, $\bb'_{i+k} \subseteq \bb_i$ for all $i \in \NN$. Then,
$$\rlim(\aa_\bullet, {\bb'_\bullet}) = \rlim(\aa_\bullet, \bb_\bullet).$$
\end{theorem}
\begin{proof}
Since $\bb_\bullet \le \bb_\bullet'$, for every $r$, we have $\rho^r(\aa_\bullet,\bb_\bullet') \le \rho^r(\aa_\bullet,\bb_\bullet).$ Therefore, we have $\rho^{\lim}(\aa_\bullet,\bb_\bullet')\le \rho^{\lim}(\aa_\bullet,\bb_\bullet).$ If $\rho^{\lim}(\aa_\bullet,\bb_\bullet')=\infty$, then we are done. So we assume that $\rho^{\lim}(\aa_\bullet,\bb_\bullet')< \infty.$ This implies that for $r\gg 1$, $\rho^{r}(\aa_\bullet,\bb_\bullet')< \infty.$

Since  $\bb_\bullet \le \bb'_\bullet$, we have $\beta_s(\aa_\bullet,\bb_\bullet) \le \beta_s(\aa_\bullet,\bb'_\bullet)$ for all $s.$ Also, we have  $\beta_s(\aa_\bullet,\bb'_\bullet) \le \beta_s(\aa_\bullet,\bb_\bullet)+k$ for all $s$ as $\bb'_{i+k} \subseteq \bb_i$ for all $i \in \NN$. Now, we have the following cases:

\textsf{Case 1.} Suppose that there is a positive integer $r_0$ such that for all $r \ge r_0$, $$\rho^{r}(\aa_\bullet,\bb_\bullet) \not\in \left\{ \dfrac{s}{\beta_s(\aa_\bullet,\bb_\bullet)} ~\Big|~ \beta_s(\aa_\bullet, \bb_\bullet)< \infty \text{ and } s \ge r\right\}.$$
We claim that $\rho^{r}(\aa_\bullet,\bb_\bullet')=\rho^{r}(\aa_\bullet,\bb_\bullet)$ for all $r \ge \max\{ k,r_0\}$. Suppose that $\rho^{r}(\aa_\bullet,\bb_\bullet)=-\infty$, then $\aa_s \subseteq \bb_t$ for all $s \ge r$ and $t \in \NN$. Therefore, $\aa_s \subseteq \bb'_t$ for all $s \ge r$ and $t \in \NN$, and hence, $\rho^{r}(\aa_\bullet,\bb'_\bullet)=-\infty.$	So, assume that $\rho^{r}(\aa_\bullet,\bb_\bullet) \ge 0.$ Since $\rho^{r}(\aa_\bullet,\bb_\bullet) =\sup\left\{ \dfrac{s}{\beta_s(\aa_\bullet,\bb_\bullet)} ~\Big|~ \beta_{s}(\aa_\bullet,\bb_\bullet)< \infty \text{ and } s \ge r \right\}$, there exists a non-decreasing sequence $s_n$ of positive integers with $s_1 \ge r$ and $\beta_{s_n}(\aa_\bullet,\bb_\bullet)< \infty$ for all $n$ such that $$\displaystyle\lim_{n \to \infty} \dfrac{s_n}{\beta_{s_n}(\aa_\bullet,\bb_\bullet)} =\rho^{r}(\aa_\bullet,\bb_\bullet).$$ Consequently, we have $$\displaystyle\lim_{n \to \infty} \dfrac{s_n}{\beta_{s_n}(\aa_\bullet,\bb_\bullet)+k} =\rho^{r}(\aa_\bullet,\bb_\bullet).$$
Consider
\begin{align*} \rho^{r}(\aa_\bullet,\bb_\bullet') & = \sup\left\{ \dfrac{s}{\beta_s(\aa_\bullet,\bb_\bullet')} ~\Big|~ \beta_s(\aa_\bullet,\bb_\bullet')< \infty \text{ and } s \ge r \right\} \\ & \ge  \sup\left\{ \dfrac{s}{\beta_s(\aa_\bullet,\bb_\bullet)+k} ~\Big|~ \beta_s(\aa_\bullet,\bb_\bullet) < \infty \text{ and } s \ge r \right\} \\ & \ge \sup\left\{ \dfrac{s_n}{\beta_{s_n}(\aa_\bullet,\bb_\bullet)+k} ~\Big|~ n \in \NN \right\}  = \lim_{n \to \infty}\dfrac{s_n}{\beta_{s_n}(\aa_\bullet,\bb_\bullet)} =\rho^{r}(\aa_\bullet,\bb_\bullet).\end{align*} Therefore,  $\rho^{\lim}(\aa_\bullet,\bb_\bullet')=\rho^{\lim}(\aa_\bullet,\bb_\bullet)$.

\textsf{Case 2.} Suppose that $\rho^{r}(\aa_\bullet,\bb_\bullet) \in \left\{ \dfrac{s}{\beta_s(\aa_\bullet,\bb_\bullet)} ~\Big|~ \beta_s(\aa_\bullet,\bb_\bullet)< \infty \text{ and }s \ge r \right\}$ for infinitely many $r.$ First note that if for some $r$, $\rho^{r}(\aa_\bullet,\bb_\bullet) \in \left\{ \dfrac{s}{\beta_s(\aa_\bullet,\bb_\bullet)} ~\Big|~ \beta_s(\aa_\bullet,\bb_\bullet)< \infty \text{ and } s \ge r \right\}$, then there exists a positive integer $s$ such that $\beta_s(\aa_\bullet,\bb_\bullet)< \infty$ and $\rho^{r}(\aa_\bullet,\bb_\bullet) =\dfrac{s}{\beta_s(\aa_\bullet,\bb_\bullet)} $. Since  $\{\rho^{r}(\aa_\bullet,\bb_\bullet)\}$ is  non-increasing sequence of positive real numbers, we have  $\rho^{k}(\aa_\bullet,\bb_\bullet) =\dfrac{s}{\beta_s(\aa_\bullet,\bb_\bullet)} $ for all $r \le k \le s$. In particular, $\rho^{s}(\aa_\bullet,\bb_\bullet) =\dfrac{s}{\beta_s(\aa_\bullet,\bb_\bullet)} $. Therefore, there is a  non-decreasing sequence of positive integers $s_n$  such that for each $n$, $\beta_{s_n}(\aa_\bullet,\bb_\bullet)< \infty$ and  $\rho^{s_n}(\aa_\bullet,\bb_\bullet) = \dfrac{s_n}{\beta_{s_n}(\aa_\bullet,\bb_\bullet)}. $ Now, since $\{\rho^{r}(\aa_\bullet,\bb_\bullet)\}$ is  non-increasing sequence of positive real numbers and $\{\rho^{s_n}(\aa_\bullet,\bb_\bullet) \} $ is a non-increasing sub-sequence of $\{\rho^{r}(\aa_\bullet,\bb_\bullet)\}$, we have $$\rho^{\lim}(\aa_\bullet,\bb_\bullet)= \lim_{r \to \infty} \rho^r(\aa_\bullet,\bb_\bullet) = \lim_{n \to \infty} \rho^{s_n}(\aa_\bullet,\bb_\bullet) = \lim_{n \to \infty} \dfrac{s_n}{\beta_{s_n}(\aa_\bullet,\bb_\bullet)+k} \le \lim_{n \to \infty} \dfrac{s_n}{\beta_{s_n}(\aa_\bullet,\bb_\bullet')}.$$ Thus, $$\rho^{\lim}(\aa_\bullet,\bb_\bullet) \le \lim_{n \to \infty} \dfrac{s_n}{\beta_{s_n}(\aa_\bullet,\bb_\bullet')} \le \lim_{n \to \infty} \rho^{s_n}(\aa_\bullet,\bb_\bullet')= \lim_{r \to \infty} \rho^r(\aa_\bullet,\bb_\bullet')=\rho^{\lim}(\aa_\bullet,\bb_\bullet').$$ Hence, in both cases, we have $\rho^{\lim}(\aa_\bullet,\bb_\bullet')=\rho^{\lim}(\aa_\bullet,\bb_\bullet).$
\end{proof}

\begin{corollary}\label{cor.rholimb1equi}
Let $\aa_\bullet$ be a graded family of ideals in $S$ and $\bb_\bullet$ be a filtration of ideals in $S$ such that $\R(\overline{\bb_\bullet})$ is a finitely generated $\R(\bb_\bullet)$-module. Then,
$$\rlim(\aa_\bullet, \overline{\bb_\bullet}) = \rlim(\aa_\bullet, \bb_\bullet).$$
\end{corollary}

\begin{proof}
Since $\R(\overline{\bb_\bullet})$ is a finitely generated $\R(\bb_\bullet)$-module, it follows from the proof of \Cref{thm.asym_res_int} that there exists a positive integer $k$ such that $\overline{\bb_{i+k}} \subseteq \bb_i$ for all $i \in \NN$. Also,  $\bb_\bullet \le \overline{\bb_\bullet}$. Hence, the assertion follows from \Cref{lem.rlim}.
\end{proof}

\begin{corollary}
    \label{lem.rholimb1equi}
    Let $\aa_\bullet$ and $\bb_\bullet$ be graded families of ideals in $S$ such that $\bb_\bullet$ is $\bb$-equivalent, for some ideal $\bb \subseteq S$. Then $$\rlim(\aa_\bullet, \bb_\bullet) = \rlim(\aa_\bullet, \bb^\bullet).$$
\end{corollary}

\begin{proof}
  There is a positive integer $k$ such that $\bb_{i+k} \subseteq \bb^i \subseteq \bb_i$ for all $i\in \NN$ as $\bb_\bullet$ is $\bb$-equivalent graded family. The assertion now follows from \Cref{lem.rlim}.
\end{proof}
We now arrive at our next main result of this section.
\begin{theorem}\label{thm.rhatrlimrho}
Let $S$ be a domain as in \Cref{asym_res_ideal}. Let $\aa_\bullet$ be filtration of nonzero ideals in $S$, and $\bb_\bullet$ be a $\bb$-equivalent graded family, for some ideal $\bb \subseteq S$. Then,
$$\rhat(\aa_\bullet, \overline{\bb_\bullet})= \rlim(\aa_\bullet, \overline{\bb_\bullet})=\rhat(\aa_\bullet, \bb_\bullet) = \rlim(\aa_\bullet, \bb_\bullet).$$
\end{theorem}

\begin{proof} It follows from the definition that
$\rhat(\aa_\bullet, \overline{\bb_\bullet}) \le \rlim(\aa_\bullet, \overline{\bb_\bullet}) \le \rho(\aa_\bullet, \overline{\bb_\bullet}).$
Moreover, by \Cref{thm.b1equivalent}, we have $\rhat(\aa_\bullet, \overline{\bb_\bullet}) = \rho(\aa_\bullet, \overline{\bb_\bullet})$. Thus, we must have
$$\rhat(\aa_\bullet, \overline{\bb_\bullet}) = \rlim(\aa_\bullet, \overline{\bb_\bullet}) = \rho(\aa_\bullet, \overline{\bb_\bullet}).$$

On the other hand, since $\bb_\bullet$ is $\bb$-equivalent, there exists an integer $k \in \NN$ such that, for all $n \ge \NN$, $\bb_{n+k} \subseteq \bb^n$, whence
$$\overline{\bb_{n+k}} \subseteq \overline{\bb^n}.$$
By \cite[Proposition 5.3.4]{SH2006}, there exists an integer $k' \in \NN$ such that, for all $n > k'$,
$$\overline{\bb^{n}} \subseteq \bb^{n-k'}.$$
 Therefore, for all $n$, $$\overline{\bb_{n+k+k'}} \subseteq {\bb^n} \subseteq \bb_n \subseteq \overline{\bb_n},$$
 and hence, $\overline{\bb_\bullet}$ is also $\bb$-equivalent.  \Cref{lem.rholimb1equi} then implies that $\rlim(\aa_\bullet, {\bb_\bullet}) = \rlim(\aa_\bullet,{\bb}^{\bullet})$ and $\rlim(\aa_\bullet, \overline{\bb_\bullet}) = \rlim(\aa_\bullet,{\bb}^{\bullet})$.
The assertion now follows from \Cref{thm.b1equivalent} as $\rhat(\aa_\bullet, \overline{\bb_\bullet}) = \rhat(\aa_\bullet, \bb_\bullet)$.
\end{proof}

\begin{remark}
    \label{rmk.lowerlim}
Let $\aa_\bullet$ and $\bb_\bullet$ be graded families of ideals in $S$. For $n \in \NN$, set $$\rho_n(\aa_\bullet,\bb_\bullet) := \sup\left\{ \frac{s}{r} \Big|~ \aa_s \not\subseteq \bb_r \text{ and } s,r \ge n \right\}.$$
The following limit can be shown to exist and, thus, we can define $$\rho_{\lim}(\aa_\bullet,\bb_\bullet) :=\lim_{n \to \infty} \rho_n(\aa_\bullet,\bb_\bullet).$$
The notion of $\rho_{\lim}$ is a direct generalization of similar constructions that were investigated in \cite{BHJT,HKZ}.

Observe that $ \rho_{\text{lim}}(\aa_\bullet,\bb_\bullet) \le \rlim(\aa_\bullet,\bb_\bullet)$ and the equality holds when $\rlim(\aa_\bullet,\bb_\bullet) < \infty$. Therefore, a direct generalization of \cite[Theorem 2.1 and Lemma 2.2]{BHJT} and \cite[Theorem 2.1]{HKZ} in terms of $\rho_{\lim}(\aa_\bullet, \bb_\bullet)$, that is similar to \Cref{thm.rhatrlimrho}, can be obtained. We leave the details to the interested reader.

\end{remark}

Finally, as a direct consequence of \Cref{lem.rholimb1equi}, we recover another version of  \Cref{thm.rhatrlimrho}, that is, when $\bb_\bullet = \bb^\bullet$, but without the filtration assumption on $\aa_\bullet$.

\begin{corollary}
    \label{cor.rholimrhohat}
    Let $S$ be a domain as in \Cref{asym_res_ideal}. Suppose that $\aa_\bullet$ is a graded family and $\bb$ be a nonzero ideal in $S$. Then,
    $$\widehat{\rho}(\aa_\bullet,\bb^\bullet) = \widehat{\rho}(\aa_\bullet,\overline{\bb^\bullet}) = \rlim(\aa_\bullet,\overline{\bb^\bullet}) = \rlim(\aa_\bullet,\bb^\bullet).$$
\end{corollary}

\begin{proof}
Note that $\overline{\bb^{\bullet}}$ is a $\bb$-equivalent family. Therefore, by \Cref{lem.rholimb1equi},  $ \rlim(\aa_\bullet,\overline{\bb^\bullet}) = \rlim(\aa_\bullet,\bb^\bullet).$  Using  \Cref{lem.estimate}, the discussion after \Cref{def.rlim}, and \Cref{tech_lemm},  we have $\widehat{\rho}(\aa_\bullet,\overline{\bb^\bullet}) \le \rhat(\aa_\bullet, \bb^\bullet) \le \rho^{\lim}(\aa_\bullet, \bb^\bullet)= \rlim(\aa_\bullet,\overline{\bb^\bullet}) \le \rho(\aa_\bullet,\overline{\bb^\bullet}) = \rhat(\aa_\bullet,\overline{\bb^\bullet})$. Hence, the assertion follows.
\end{proof}

%%%%%%%%%%%%%%%%%%%%%%%%%%%%%%%%%%%%%%%%%%%%%%%%%%%%

\section{Finiteness and rationality of resurgence numbers} \label{sec.finite}

In this section, we shall discuss situations where the resurgence and asymptotic resurgence numbers are finite and rational.
Note that, by \Cref{ex.irrational}, any positive real number can be realized as the resurgence or asymptotic resurgence of a pair of graded families of ideals in $S$. Main results in this section are stated in Theorems \ref{linearly_smaller} and \ref{thm.rhoRat}; see also \Cref{cor.rhoRatLim}.

Observe that, for any graded families $\aa_\bullet$ and $\bb_\bullet$ of ideals in $S$, by \Cref{lem.estimate}, $\rho(\aa_\bullet,\bb_\bullet) \leq \rho(\aa_\bullet,\bb_1^\bullet)$. Thus, if $\rho(\aa_\bullet,\bb_1^\bullet)< \infty$, then $\rho(\aa_\bullet,\bb_\bullet)< \infty.$ We shall prove that when $\bb_\bullet$ is a Noetherian filtration, the converse also holds.

\begin{theorem}
Let $\aa_\bullet$ be a graded family and let $\bb_\bullet$ be a Noetherian filtration of ideals in $S$. Then $\rho(\aa_\bullet,\bb_1^\bullet)< \infty$ if and only if $\rho(\aa_\bullet,\bb_\bullet)< \infty.$
\end{theorem}
\begin{proof}
Assume that $\rho(\aa_\bullet,\bb_\bullet)< \infty.$ Since $\bb_\bullet$ is a Noetherian filtration, by \cite[Proposition 2.1]{Sch88},  there exists a positive integer $k$ such that $\R^{[k]}(\aa_\bullet)$ is a standard graded $S$-algebra. Now, by Lemma \ref{res_as_prod}, $\rho(\aa_\bullet,\bb_k^{\bullet}) \le k \rho(\aa_\bullet,\bb_\bullet)$. Therefore,  $\rho(\aa_\bullet,\bb_k^{\bullet})<\infty$. Let $s,r$ be positive integers such that $\aa_s \not\subseteq \bb_1^r$.  As $\bb_k \subseteq \bb_1$, we get that $\aa_s \not\subseteq \bb_k^r$. Hence, by definition, $\rho(\aa_\bullet,\bb_1^{\bullet}) \le \rho(\aa_\bullet,\bb_k^{\bullet}) < \infty.$
\end{proof}

The resurgence number takes $-\infty$ value in a very special case, as seen in the next lemma.

\begin{lemma}\label{res_-infty}
	Let $\aa_\bullet $ and $\bb_\bullet $ be filtration of ideals in $S$. Then, $\rho(\aa_\bullet, \bb_\bullet) = -\infty$ if and only if $\aa_1 \subseteq \bigcap_{i\ge 1} \bb_i$.
\end{lemma}

\begin{proof}
	It can be seen that $\rho(\aa_\bullet, \bb_\bullet) = -\infty$ if and only if $\aa_s \subseteq \bb_r$ for all $s, r \in \NN$. This is the case if and only if $\aa_1 \subseteq \bigcap_{i \ge 1} \bb_i$.
\end{proof}

An example when the condition in \Cref{res_-infty} is satisfied is when $\bb_i = S$ for all $i \ge 0$. The following observations follow immediately from \Cref{res_-infty}.

\begin{remark}
	Let $\aa_\bullet $ and $\bb_\bullet $ be filtration of ideals in $S$.
	\begin{enumerate}
		\item If  $\bigcap_{i\geq 1}\bb_i =(0)$, then $\rho(\aa_\bullet,\bb_\bullet) =-\infty$ if and only if $\aa_i =0$ for all $i\geq 1$.
		\item If $\rho(\aa_\bullet,\bb_\bullet) \neq -\infty$, then $0< \rho(\aa_\bullet,\bb_\bullet).$
		 \item If  $\bigcap_{i\geq 1}\bb_i =(0)$ and $\aa_1 \neq (0)$, then $\rho(\aa_\bullet,\bb_\bullet) \neq -\infty$, and hence, $0< \rho(\aa_\bullet,\bb_\bullet).$
		 %\item The condition that $\aa_1 \not\subseteq \bigcap_{i \ge 1} \bb_i$ is satisfied, for instance, when $\aa_1 \not= (0)$ and $\tau_\bb$ is separated.
	 \end{enumerate}
\end{remark}

\begin{remark}\label{res_-infty-infty}
	Let $\aa_\bullet $ and $\bb_\bullet $ be filtration of ideals in $S$.  If $\rho(\aa_\bullet,\bb_\bullet) =-\infty$ and $\rho(\bb_\bullet,\aa_\bullet) < \infty$, then $\bb_j=\aa_i=\aa_1$ for all $i\geq 1$ and $j\gg 1$.
\end{remark}

\begin{proof} Since $\rho(\aa_\bullet,\bb_\bullet) =-\infty$, by \Cref{res_-infty}, $\aa_1 \subseteq \bigcap_{i \ge 1} \bb_i$. Let  $k \in \mathbb{N}$ be smallest positive integer such that $k > \rho(\bb_\bullet,\aa_\bullet)$. We have $\bb_{ki} \subseteq \aa_i$ for every $i \geq 1$.  Therefore, for every $i\geq 1$, $\bb_{ki} \subseteq \aa_i \subseteq \aa_1 \subseteq \cap_{j\geq 1}\bb_j \subseteq \bb_{ki}$. This  implies that $\bb_{ki} =\aa_i=\aa_1$ for all $i \ge 1$. It further follows that $\bb_j=\aa_i=\aa_1$ for all $i \geq 1$ and $j \geq k$.
\end{proof}

In studying the finiteness of resurgence numbers, we shall make use of the topology defined by a filtration of ideals.

\begin{definition}
A filtration $\aa_\bullet $, with $\aa_0=S$, defines a topology on the additive group $(S,+)$, which we shall denote by $\tau_\aa$. Particularly, the open neighborhoods of any $x \in S$ is given by $\{x+\aa_i\}_{i\ge 0}$. This makes $(S,+)$ a topological group. We say that the topology $\tau_{\aa}$ is \emph{separated} (or \emph{Hausdorff}) if $\bigcap_{i\ge 0} \aa_i =(0)$. Equivalently, $\tau_{\aa}$ is separated if and only if $\bigcap_{i\gg 1} \aa_i =(0).$
\end{definition}

In general, $\aa_\bullet$ defines a \emph{finer} topology than $\bb_\bullet$ does if, for all $i \ge 1$, there exists a non-negative integer $f_i$ such that $\aa_{f_i} \subseteq \bb_i$. We shall use a slightly stronger notion to compare the topology given by $\aa_\bullet$ and $\bb_\bullet$.

\begin{definition}
	\label{def.top}
	Let $\aa_\bullet $ and $\bb_\bullet $ be filtration of ideals in $S$.
	\begin{enumerate}
		\item The topology $\tau_\aa$ defined by $\aa_\bullet$ is said to be \emph{linearly finer} than the topology $\tau_\bb$ given by $\bb_\bullet$ if there exists a \emph{linear} function $f \in \ZZ_{\ge 0}[x]$ such that for all $i \ge 0$, we have $\aa_{f(i)} \subseteq \bb_i$. In this case, we also say that the topology $\tau_\bb$ is \emph{linearly coarser} than $\tau_\aa$.
		\item The topology $\tau_\aa$ and $\tau_\bb$ are said to be \emph{linearly equivalent} if $\tau_\aa$ is linearly finer than $\tau_\bb$ and $\tau_\bb$ is also linearly finer than $\tau_\aa$.
	\end{enumerate}
\end{definition}

The following result characterizes pairs of filtration whose resurgence number is a finite number.

\begin{theorem}\label{linearly_smaller}
	Let $\aa_\bullet $ and $\bb_\bullet $ be filtration of ideals in $S$. Then, $\tau_{\aa}$ is linearly finer than $\tau_{\bb}$ if and only if  $\rho(\aa_\bullet,\bb_\bullet) <\infty.$
\end{theorem}

\begin{proof}
	Suppose first that $\tau_{\aa}$ is linearly finer than $\tau_{\bb}$. Then, there exists a linear function $f: \mathbb{N} \to \mathbb{N}$, say $f(n)=an+b$, such that  $\aa_{f(i)} \subseteq \bb_i$ for every $i \ge 1$.  Let $s,r \in \NN$ be such that $\dfrac{s}{r}>a+|b|.$ Clearly, $s > r(a+|b|) \geq ar+b=f(r)$, which implies that $\aa_s \subseteq \aa_{f(r)} \subseteq \bb_r$. Thus, $\rho(\aa_\bullet,\bb_\bullet) \leq a+|b|<\infty.$
	
	Conversely, suppose that $\rho(\aa_\bullet,\bb_\bullet)<\infty.$ If $\rho(\aa_\bullet,\bb_\bullet) = -\infty,$ then $\aa_i \subseteq \bb_i$ for all $i \geq 1.$ Thus, $\tau_{\aa}$ is linearly finer than $\tau_{\bb}$. Assume that $\rho(\aa_\bullet, \bb_\bullet) > 0$. Define $f: \NN \to \NN$ by $f(n)=\ceil{\rho(\aa_\bullet,\bb_\bullet)}n+1$ for all $n \geq 1$. Then, for every $i\geq 1$, $\aa_{f(i)} =\aa_{\ceil{\rho(\aa_\bullet,\bb_\bullet)}i+1} \subseteq \bb_i$ as $\dfrac{f(i)}{i} > \rho(\aa_\bullet,\bb_\bullet)$. Hence, $\tau_{\aa}$ is linearly finer than $\tau_{\bb}$.
\end{proof}

As an immediate consequence of \Cref{linearly_smaller}, we obtain the following result on the resurgence number of pairs of filtration that define linearly equivalent topology.

\begin{corollary}\label{linearly_equi}
	Let $\aa_\bullet $ and $\bb_\bullet $ be filtration of ideals in $S$.  Then,  the topology $\tau_{\aa}$ and $\tau_\bb$ are linearly equivalent if and only if   $\rho(\aa_\bullet,\bb_\bullet) <\infty$ and $\rho(\bb_\bullet,\aa_\bullet) < \infty.$
\end{corollary}

\begin{proof}
	The conclusion follows from \Cref{linearly_smaller}.
\end{proof}

\begin{example}
Let $\mathfrak{p}$ be a prime ideal in $S$ and $I$ be a $\mathfrak{p}$-primary ideal. Let $k$ be the smallest positive integer such that $\mathfrak{p}^k \subseteq I$ and $l$ be the largest positive integer such that $I\subseteq \mathfrak{p}^l.$ Consider the graded families
$$\aa_\bullet = I^\bullet \text{ and } \bb_\bullet=\mathfrak{p}^\bullet.$$

Since $\bb_{ki}\subseteq \aa_i$ and $\aa_i \subseteq \bb_i$ for all $i$, we get that  the topology $\tau_{\aa}$ and $\tau_\bb$ are linearly equivalent. Therefore, by Corollary \ref{linearly_equi},  $\rho(\aa_\bullet,\bb_\bullet) <\infty$ and $\rho(\bb_\bullet,\aa_\bullet)< \infty$.

We shall see that $\dfrac{1}{l+1} \le \rho(\aa_\bullet,\bb_\bullet) \leq \dfrac{1}{l}$ and $k-1 \le \rho(\bb_\bullet,\aa_\bullet) \leq k.$ Indeed, since $I \not\subseteq \mathfrak{p}^{l+1}$ and $\mathfrak{p}^{k-1} \not \subseteq I$, we have $\dfrac{1}{l+1} \le \rho(\aa_\bullet,\bb_\bullet)$ and $k-1 \le \rho(\bb_\bullet,\aa_\bullet).$

On the other hand, let $s,r$ be positive integers such that $r \le sl$. Then, $I^s \subseteq \mathfrak{p}^{ls} \subseteq \mathfrak{p}^r,$ i.e., $\aa_s \subseteq \bb_r$ if $r \le sl$. Thus, the upper bound for $\rho(\aa_\bullet,\bb_\bullet)$ follows. Similarly, let $s,r$ be positive integers such that $rk \le s$. Then, $\mathfrak{p}^s \subseteq \mathfrak{p}^{kr} \subseteq I^r,$ i.e., $\bb_s \subseteq \aa_r$ if $kr \le s$. Thus, the upper bound for $\rho(\bb_\bullet,\aa_\bullet)$ holds.
\end{example}

\begin{remark} We will see, as a consequence of \Cref{cor.rhoRatLim} below, that when $S$ is a domain as in \Cref{asym_res_ideal}, $\aa_\bullet$ is a filtration and $\bb_\bullet$ is a $\bb$-equivalent graded family, $\rhat(\aa_\bullet, \bb_\bullet) < \infty$ if and only if $\rho(\aa_\bullet, \bb_\bullet) < \infty$.
\end{remark}

\begin{corollary}
	\label{cor.rationalR}
	Let $S$ be an analytically unramified local ring. Let $\aa_\bullet$ and $\bb_\bullet$ be as in \Cref{cor.rationalAR}. Suppose further that $\bb_\bullet$ is $\bb$-equivalent, for some ideal $\bb \subseteq S$. Then, $\rho(\aa_\bullet, \bb_\bullet)$ is either infinity or a rational number.
\end{corollary}

\begin{proof}
	If $\rhat(\aa_\bullet, \overline{\bb_\bullet}) = \rho(\aa_\bullet, \bb_\bullet)$ then the assertion follows from \Cref{cor.rationalAR}. If $\rhat(\aa_\bullet, \overline{\bb_\bullet}) < \rho(\aa_\bullet, \bb_\bullet)$ then the assertion follows from \Cref{thm.rhoReesVal} and \Cref{rmk.regularRing}.
\end{proof}

%%%%%%%%%%%%%%%%%%%%%%%%%%%%%%%%%%%%%%%%%%%%%%%%%%%%%%%%%

We continue to our final results on the rationality of resurgence number with a condition in terms of $\rlim$. This result is new even in the standard case of filtration of symbolic and ordinary powers.

\begin{theorem}\label{thm.rhoRat}
Let $\aa_\bullet$ and $\bb_\bullet$ be graded families of ideals in $S$. If $\rlim(\aa_\bullet, \bb_\bullet) \not= \rho(\aa_\bullet, \bb_\bullet)$, then $\rho(\aa_\bullet, \bb_\bullet)$ is a rational number.
\end{theorem}

\begin{proof} We first claim that $\rlim(\aa_\bullet, \bb_\bullet) < \infty$ if and only if $\rho(\aa_\bullet, \bb_\bullet) < \infty.$ One implication is obvious. Suppose that $\rlim(\aa_\bullet, \bb_\bullet) < \infty$. Take any $M \in \RR$ with $\rlim(\aa_\bullet, \bb_\bullet) < M$. By definition, there exists a positive integer $n_0$ such that $\rho^n(\aa_\bullet, \bb_\bullet)<M$ for all $n \ge n_0$.
Consider any $s,r \in \NN$ such that $\aa_s \not\subseteq \bb_r$. If $s \ge n_0$ then, by the definition of $\rho^{n_0}(\aa_\bullet, \bb_\bullet)$, we have $\dfrac{s}{r} \le \rho^{n_0}(\aa_\bullet, \bb_\bullet)<M.$ On the other hand, if $s < n_0$ then $\dfrac{s}{r} \le n_0$. Thus,
$$\rho(\aa_\bullet, \bb_\bullet) \le \max\{n_0, M\}< \infty.$$

Now, if $\rlim(\aa_\bullet, \bb_\bullet) < \rho(\aa_\bullet, \bb_\bullet)$ then it follows from our claim that $\rho(\aa_\bullet, \bb_\bullet) < \infty$. Set
$$\theta = \rho(\aa_\bullet, \bb_\bullet) - \rlim(\aa_\bullet, \bb_\bullet) > 0.$$
Since $\displaystyle\lim_{n \rightarrow \infty} \rho^n(\aa_\bullet, \bb_\bullet) = \rlim(\aa_\bullet, \bb_\bullet) = \rho(\aa_\bullet, \bb_\bullet) - \theta$, there exists $n_1 \in \NN$ such that, for all $n \ge n_1$, $\rho^n(\aa_\bullet, \bb_\bullet) < \rho(\aa_\bullet, \bb_\bullet)-\frac{\theta}{2}.$
This implies that, for $s \ge n_1$, if $\beta_s(\aa_\bullet, \bb_\bullet) < \infty$, then we have
$$\dfrac{s}{\beta_s(\aa_\bullet, \bb_\bullet)} < \rho(\aa_\bullet, \bb_\bullet) - \frac{\theta}{2}.$$
Hence, together with \Cref{eq.rhosup}, it follows that
$$\rho(\aa_\bullet, \bb_\bullet) = \max  \left\{\dfrac{s}{\beta_s(\aa_\bullet, \bb_\bullet)}~\Big|~ 1 \le s <n_1 \text{ and } \beta_s(\aa_\bullet, \bb_\bullet) < \infty \right\}.$$
Particularly, $\rho(\aa_\bullet, \bb_\bullet)$ is a rational number.
\end{proof}

As an immediate consequence of \Cref{thm.rhoRat} and  \Cref{thm.rhatrlimrho} we obtain the following results.

\begin{corollary} \label{cor.rhoRatLim}
Let $S$ be a domain as in \Cref{asym_res_ideal}. Let $\aa_\bullet$ be a filtration and let $\bb_\bullet$ be a graded family of nonzero ideals in $S$. Suppose that $\bb_\bullet$ is $\bb$-equivalent for some ideal $\bb \subseteq S$. If $\rhat(\aa_\bullet, \overline{\bb_\bullet}) \not= \rho(\aa_\bullet, \bb_\bullet)$, then $\rho(\aa_\bullet, \bb_\bullet)$ is a rational number.
\end{corollary}

\begin{corollary} \label{cor.rhoRat}
Let $S$ be as in \Cref{asym_res_ideal} and let $I \subseteq S$ be a nonzero proper ideal. Then,
\begin{enumerate}
\item $\rho(\overline{I^\bullet}, I^\bullet)$ is a rational number.
\item if $\rhat(I) \not= \rho(I)$, then $\rho(I)$ is a rational number.
\end{enumerate}
\end{corollary}

\begin{proof} The assertion is a direct consequence of \Cref{cor.rhoRatLim}, noticing that the family $\{I^i\}_{i \ge 1}$ is $I$-equivalent.
\end{proof}

We end the paper with the following general question.

\begin{question}
	Characterize for which pairs of graded families $(\aa_\bullet, \bb_\bullet)$, the resurgence and asymptotic resurgence numbers, $\rho(\aa_\bullet, \bb_\bullet)$ and $\rhat(\aa_\bullet, \bb_\bullet)$, are rational.
\end{question}

%%%%%%%%%%%%%%%%%%%%%%%%%%%%%%%%%%
%%%%%%%%%%%%%%%%%%%%%%%%%%%%

%%%%%%%%%%%%%%%%%%%%%%%%%%%%%%%%%

%%%%%%%%%%%%%%%%%%%%%%%%%%%%%%%%%%

\end{document}